\documentclass[bibalpha]{amsart}
\usepackage{amssymb}
\usepackage{textcomp}
\usepackage{enumitem}
\usepackage{comment}
\usepackage{fdsymbol}
\usepackage[mathscr]{euscript}

\newtheorem{theorem}{Theorem}[section]

\newtheorem{thm}[theorem]{Theorem}
\newtheorem{prop}[theorem]{Proposition}
\newtheorem{lem}[theorem]{Lemma}

\newtheorem{fact}[theorem]{Fact}
\newtheorem{cor}[theorem]{Corollary}

\newtheorem{conj}{Conjecture}

\theoremstyle{definition}

\theoremstyle{remark}

\newcommand{\rvec}{\R_{\mathrm{Vec}}}

\newcommand{\down}{\mathbf{D}}
\newcommand{\up}{\mathbf{U}}

\ProvideTextCommandDefault{\cprime}{(U+042C)}

\newcommand{\la}{\langle}
\newcommand{\ra}{\rangle}

\newcommand{\gr}{{\operatorname{gr}}}
\newcommand{\st}{\operatorname{st}}
\newcommand{\valp}{\operatorname{Val_p}}
\newcommand{\valq}{\operatorname{Val_q}}

\newcommand{\frp}{\mathfrak{p}_\lambda}
\newcommand{\zca}{(\Z,+,C_\alpha)}
\newcommand{\zcb}{(\Z,+,C_\beta)}
\newcommand{\shc}{(\Z,+,C_\alpha)^{\mathrm{Sh}}}

\newcommand{\rzc}{(\rz,+,C)}
\newcommand{\zp}{(\Z_p,+,\valp)}
\newcommand{\zq}{(\Z_q,+,\valq)}
\newcommand{\zpp}{(\Z,+,\valp)}

\newcommand{\Th}{\mathrm{Th}}
\newcommand{\gh}{\mathrm{GH}}

\newcommand{\lab}{L_{\mathrm{Ab}}}
\newcommand{\lval}{L_{\mathrm{Val}}}
\newcommand{\linduce}{L_{\mathrm{In}}}

\newcommand{\acl}{\operatorname{acl}}

\newcommand{\cl}{\operatorname{Cl}}
\newcommand{\bd}{\operatorname{Bd}}

\newcommand{\Sh}[1]{\ensuremath{\mathscr{#1}^{\mathrm{Sh}}}}
\newcommand{\Eq}[1]{\ensuremath{\mathscr{#1}^{\mathrm{eq}}}}
\newcommand{\Sq}[1]{\ensuremath{\mathscr{#1}^{\square}}}
\newcommand{\nip}{\mathrm{NIP}}

\newcommand{\minf}{\mathbf{Inf}}
\newcommand{\mfin}{\mathbf{Fin}}
\newcommand{\jac}{\mathbf{d}}

\newcommand{\Cal}[1]{\ensuremath{\mathcal{#1}}}
\newcommand{\Sa}[1]{\ensuremath{\mathscr{#1}}}

\newcommand{\lam}{\ensuremath{\lambda}}

\newcommand{\F}{\mathbb{F}}
\newcommand{\Z}{\mathbb{Z}}
\newcommand{\N}{\mathbb{N}}
\newcommand{\C}{\mathbb{C}}
\newcommand{\Q}{\mathbb{Q}}
\newcommand{\R}{\mathbb{R}}
\newcommand{\E}{\mathbb{E}}
\newcommand{\bS}{\mathbb{S}}
\newcommand{\G}{\mathbb{G}}
\newcommand{\bH}{\mathbb{H}}

\newcommand{\rz}{\R/\Z}
\newcommand{\rfield}{(\R,+,\times}
\newcommand{\pring}{(\Z_p,+,\times}
\newcommand{\pfield}{(\Q_p,+,\times}
\newcommand{\pexp}{\operatorname{Exp}}
\newcommand{\plog}{\operatorname{Log}}
\newcommand{\roag}{(\R,+,<)}

\begin{document}
\title[]{Dp-minimal expansions of $(\Z,+)$ via dense pairs via Mordell-Lang}

\author{Erik Walsberg}
\address{Department of Mathematics, Statistics, and Computer Science\\
Department of Mathematics\\University of California, Irvine, 340 Rowland Hall (Bldg.\# 400),
Irvine, CA 92697-3875}
\email{ewalsber@uci.edu}
\urladdr{http://www.math.illinois.edu/\textasciitilde erikw}

\date{\today}

\maketitle

\begin{abstract}
This is a contribution to the classification problem for dp-minimal expansions of $(\Z,+)$.
Let $S$ be a dense cyclic group order on $(\Z,+)$.
We use results on ``dense pairs" to construct uncountably many dp-minimal expansions of $(\Z,+,S)$.
These constructions are applications of the Mordell-Lang conjecture and are the first examples of ``non-modular" dp-minimal expansions of $(\Z,+)$.
We canonically associate
an o-minimal expansion $\Sa R$ of $\rfield)$, an $\Sa R$-definable circle group $\bH$, and a character $\Z \to \bH$ to a ``non-modular" dp-minimal expansion of $(\Z,+,S)$.
We also construct a ``non-modular" dp-minimal expansion of $(\Z,+,\valp)$ from the character $\Z \to \Z^\times_p, k \mapsto \pexp(pk)$.
\end{abstract}

\section{Introduction}
\noindent
We construct new dp-minimal expansions of $(\Z,+)$ and take some steps towards classifying dp-minimal expansions of $(\Z,+)$ which define either a dense cyclic group order or a $p$-adic valuation.
(Every known proper dp-minimal expansion of $(\Z,+)$ defines either a dense cyclic group order, a $p$-adic valuation, or $<$.)
\newline

\noindent
We recall the definition of dp-minimality in Section~\ref{section:prelim}.
Dp-minimality is a strong form of $\nip$ which is broad enough to include many interesting structures and narrow enough to have very strong consequences.
O-minimality and related notions imply dp-minimality. Johnson~\cite{Johnson-thesis} classified dp-minimal fields.
Simon~\cite{Simon-dp} showed that an expansion of $(\R,+,<)$ is dp-minimal if and only if it is o-minimal.
We summarize recent work on dp-minimal expansions of $(\Z,+)$ in Section~\ref{section:survey}.
\newline

\noindent
It was an open question for some years whether every proper dp-minimal expansion of $(\Z,+)$ is interdefinable with $(\Z,+,<)$ \cite[Question 5.32]{toomanyII}.
It turns out that this question was essentially answered before it was posed, in work on ``dense pairs".
We will show, applying work of Hieronymi and G\"{u}naydin~\cite{HiGu}, that 
if $\bS$ is the unit circle, $t \in \R$ is irrational, and $\chi : \Z \to \bS$ is the character $\chi(k) := e^{2\pi i tk}$ then the structure induced on $\Z$ by $\rfield)$ and $\chi$ is dp-minimal.
\newline

\noindent
Indeed, for every known dp-minimal expansion $\Sa Z$ of $(\Z,+)$ there is a dp-minimal field $\Sa K$, a semiabelian $\Sa K$-variety $V$, and a character $\chi : \Z \to V(\Sa K)$ such that the structure $\Sa Z_\chi$ induced on $\Z$ by $\Sa K$ and $\chi$ is dp-minimal and $\Sa Z$ is a reduct of $\Sa Z_\chi$.
\newline

\noindent
We now briefly describe how the known dp-minimal expansions of $(\Z,+)$ fall into this framework.
It follows directly from the Mordell-Lang conjecture that if $\beta \in \C^\times$ is not a root of unity then the structure induced on $\Z$ by $(\C,+,\times)$ and the character $k \mapsto \beta^k$ is interdefinable with $(\Z,+)$.
It follows from a result of Tychonievich~\cite[Theorem 4.1.2]{Tychon-thesis} that if $\beta \in \R^\times \setminus \{-1,1\}$ then the structure induced on $\Z$ by $\rfield)$ and the character $k \mapsto \beta^k$ is interdefinable with $(\Z,+,<)$.
(It is also shown in \cite{Ma-adic} that  if $\beta 
\in \Q^\times_p$ and $\valp(\beta) \neq 0$ then $(\Z,+,<)$ is interdefinable with the structure induced on $\Z$ by $\pfield)$ and $k \mapsto \beta^k$.)
Below we apply work of Mariaule~\cite{Ma-adic} to show that there is $\beta \in 1 + p\Z_p$ such that the structure induced on $\Z$ by $\pfield)$ and $k \mapsto \beta^k$ is a dp-minimal expansion of $\zpp$.
The only other previously known dp-minimal expansion of $(\Z,+)$ is $(\Z,+,S)$ where $S$ is a dense cyclic group order~\cite{TW-cyclic}.
There is a unique $\beta \in \bS$ such that $S$ is the pullback of the clockwise cyclic order on $\bS$ by $k \mapsto \beta^k$.
So the structure induced on $\Z$ by $\rfield)$ and $k \mapsto \beta^k$ is a dp-minimal expansion of $(\Z,+,S)$.
\newline

\noindent
We produce uncountably many new dp-minimal expansions of $(\Z,+,S)$.
Let $\E$ be an elliptic curve defined over $\R$, $\E^0(\R)$ be the connected component of the identity, and $\chi : \Z \to \E^0(\R)$ be a character such that $S$ is the pullback by $\chi$ of the natural cyclic order on $\E^0(\R)$.
We apply \cite{HiGu} to show that the structure $\Sa Z_\E$ induced on $\Z$ by $\rfield)$ and $\chi$ is a proper dp-minimal expansion of $(\Z,+,S)$.
We also show that $\E^0(\R)$ may be recovered up to semialgebraic isomorphism from $\Sa Z_\E$.
It follows that there is an uncountable family of dp-minimal expansions of $(\Z,+,S)$ no two of which are interdefinable.
\newline

\noindent
We describe how $\E^0(\R)$ may be recovered from $\Sa Z_\E$.
Let $C$ be the usual clockwise cyclic order on $\rz$.
Given any dp-minimal expansion $\Sa Z$ of $(\Z,+,S)$ we define a completion $\Sq Z$ of $\Sa Z$, this $\Sq Z$ is an o-minimal expansion of $\rzc$ canonically associated to $\Sa Z$.
We show that $\Sq Z_\E$ is the structure induced on $\rz$ by $\rfield)$ and the unique (up to sign) topological group isomorphism $\rz \to \E^0(\R)$.
The recovery of $\E^0(\R)$ from $\Sq Z_\E$ is a special case of a canonical correspondence between
\begin{enumerate}
    \item non-modular o-minimal expansions $\Sa C$ of $\rzc$, and
    \item pairs $\langle \Sa R,\bH \rangle$ where $\Sa R$ is an o-minimal expansion of $\Sa R$ and $\bH$ is an $\Sa R$-definable circle group.
\end{enumerate}
Given $\langle \Sa R,\bH \rangle$, $\Sa C$ is unique up to interdefinibility.
Given $\Sa C$, $\Sa R$ is unique up to interdefinibility and $\bH$ is unique up to $\Sa R$-definable isomorphism.
\newline

\noindent
We describe $\Sq Z$ for a fixed dp-minimal expansion $\Sa Z$ of $(\Z,+,S)$.
Let $\psi : \Z \to \rz$ be the unique character such that $S$ is the pullback of $C$ by $\psi$.
Let $\Sa Z \prec \Sa N$ be highly saturated, $\Sh N$ be the Shelah expansion of $\Sa N$, and $\minf$ be the natural subgroup of infinitesimals in $\Sa N$.
We identify $N/\minf$ with $\rz$ and identify the quotient map $N \to \rz$ with the standard part map.
As $\minf$ is $\Sh N$-definable we regard $\rz$ as an imaginary sort of $\Sh N$.
A slight adaptation of \cite{big-nip} shows that the following structures are interdefinable:
\begin{enumerate}
\item The structure on $\rz$ with an $n$-ary relation defining the closure in $(\rz)^n$ of $\{ (\psi(a_1),\ldots,\psi(a_n)) : (a_1,\ldots,a_n) \in X \}$ for each $\Sh Z$-definable $X \subseteq \Z^n$,
\item The structure on $\rz$ with an $n$-ary relation defining the image under the standard part map $N^n \to (\rz)^n$ of each $\Sa N$-definable subset of $N^n$,
\item The structure induced on $\rz$ by $\Sh N$.
\end{enumerate}
We refer to any of these structure as $\Sq Z$.
It follows from $(3)$ that $\Sq Z$ is dp-minimal, a slight adaptation of \cite{Simon-dp} shows that any dp-minimal expansion of $\rzc$ is o-minimal, so $\Sq Z$ is o-minimal.
We will see that the structure induced on $\Z$ by $\Sq Z$ and $\psi$ is a reduct of the Shelah expansion of $\Sa Z$.
In future work we intend to show that these two are interdefinable.
This will reduce the question ``what are the dp-minimal expansions of $(\Z,+,S)$ to ``for which o-minimal expansions $\Sa C$ of $\rzc$ is the structure induced on $\Z$ by $\Sa C$ and $\psi$ dp-minimal"?
\newline

\noindent
We also define an analogous completion $\Sq P$ of a dp-minimal expansion $\Sa P$ of $\zpp$, this $\Sq P$ is a dp-minimal expansion of $\zp$.
The structure induced on $\Z$ by $\Sq P$ is reduct of $\Sh P$.
We expect the induced structure to be interdefinable with $\Sh P$.
\newline


\noindent
It is easy to see that $\Sq Z_\E$ defines an isomorphic copy of $\rfield)$.
It follows that if $\Sa Z_\E \prec \Sa N_\E$ is highly saturated then the Shelah expansion of $\Sa N_\E$ interprets $\rfield)$, so $\Sa Z_\E$ should be ``non-modular".
(One can show that $\Sa Z_\E$ itself does not interpret an infinite field.)
At present there is no published notion of modularity for general $\nip$ structures, but there should be a notion of modularity for $\nip$ (or possibly just distal) structures which satisfies the following.
\begin{enumerate}[label=(A\arabic*)]
    \item A modular structure cannot interpret an infinite field.
    \item Abelian groups, linearly (or cyclically) ordered abelian groups, $\nip$ valued abelian groups, and ordered vector spaces are modular.
    \item If $\Sa M$ is modular and the structure induced on $A \subseteq M^n$ by $\Sa M$ eliminates quantifiers then the induced structure is modular.
    In particular the Shelah expansion of a modular structure is modular.
    (Recall that the induced structure eliminates quantifiers if and only if every definable subset of $A^m$ is of the form $A^m \cap X$ for $\Sa M$-definable $X$.)
    \item An o-minimal structure is modular if and only if it does not define an infinite field.
    (This should follow from the Peterzil-Starchenko trichotomy.)
\end{enumerate}
In this paper we will assume that there is a notion of modularity satisfying these conditions, but none of our results fail if this is not true.
$($A$2)$ implies that all previously known dp-minimal expansions of $(\Z,+)$ are modular.
$($A$1)$ and $($A$3)$ imply that if $\Sq Z$ defines $\rfield)$ then $\Sa Z$ is non-modular.
If $\Sq Z$ does not define $\rfield)$ then $($A$4)$ implies that $\Sq Z$ is modular.
We expect that if $\Sq Z$ is modular then $\Sa Z$ is modular.
\newline

\noindent
We will see that if $\Sa P$ is the structure induced on $\Z$ by $\pfield)$ and the character $k \mapsto \pexp(pk)$ then $\Sq P$ is interdefinable with the structure induced on $\Z_p$ by $\pfield)$ and the isomorphism $(\Z_p,+) \to (1 + p\Z_p,\times), a \mapsto \pexp(pa)$.
It follows that the Shelah expansion of a highly saturated $\Sa P \prec \Sa N$ interprets $\pfield)$, so $\Sa P$ is non-modular.
We again expect that $\Sa P$ is modular if and only if $\Sq P$ is modular, but we do not have a modular/non-modular dichotomy for dp-minimal expansions of $\zp$ (we lack a $p$-adic Peterzil-Starchenko.)
It seems reasonable to conjecture that a dp-minimal expansion of $\zp$ is non-modular if and only if it defines an isomorphic copy of $\pfield)$.
\newline

\noindent
We now summarize the sections.
In Section~\ref{section:prelim} we recall some background model-theoretic notions, in Section~\ref{section:cyclic-orders} we recall background on cyclically ordered abelian groups, and in Section~\ref{section:def-groups} we recall some basic facts on definable groups in o-minimal expansions of $\rfield)$.
In Section~\ref{section:survey} we survey previous work on dp-minimal expansions of $(\Z,+)$.
In Section~\ref{section:examples} we construct new dp-minimal expansions of $(\Z,+,S)$ where $S$ is a dense cyclic group order.
In Section~\ref{section:completion} we describe the o-minimal completion of a strongly dependent expansion of $(\Z,+,S)$.
We also show that the Shelah expansion $(\Z,+,S)^{\mathrm{Sh}}$ of $(\Z,+,S)$ is interdefinable with the structure induced on $\Z$ by $\rzc$ and $\psi$, where $\psi : \Z \to \rz$ is the unique character such that $S$ is the pullback of $C$ by $\psi$.
It follows that $(\Z,+,S)^{\mathrm{Sh}}$ is a reduct of each of our dp-minimal expansions of $(\Z,+,S)$.
In Section~\ref{section:interdef} we show that two of our dp-minimal expansions of $(\Z,+,S)$ are interdefinable if and only if the associated semialgebraic circle groups are semialgebraically isomorphic.
In Section~\ref{section:p-adic} we construct a new dp-minimal expansion $\Sa P$ of $\zpp$ and in Section~\ref{section:p-adic-completion} we describe the $p$-adic completion of a dp-minimal expansion of $\zpp$.
In Section~\ref{section:elliptic} we give a conjecture which implies that one can construct uncountably many dp-minimal expansions of $\zpp$ from $p$-adic elliptic curves.
Finally, in Section~\ref{section:gen-ques} we briefly discuss the question of whether our completion constructions are special cases of an abstract model-theoretic completion.

\subsection{Acknowledgements}
Thanks to Philipp Hieronymi for various discussions on dense pairs and thanks to the audience of the Berkeley logic seminar for showing interest in a talk that turned into this paper.
This paper owes a profound debt to Minh Chieu Tran.
He proposed that if $\E$ is an elliptic curve defined over $\Q$ and the group $\E(\Q)$ of $\Q$-points of $\E$ is isomorphic to $(\Z,+)$, then the structure induced on $\E(\Q)$ by $\pfield)$ might be dp-minimal and that one might thereby produce a new dp-minimal expansion of $(\Z,+)$.
Conjecture~\ref{conj:tate} is a modification of this idea (there does not appear to be anything to be gained by restricting to $\E(\Q)$ as opposed to other infinite cyclic subgroups of $\E(\Q_p)$, or by assuming that $\E$ is defined over $\Q$ as opposed to $\Q_p$.)

\section{Conventions, notation, and terminology}
\noindent
Given a tuple $x = (x_1,\ldots,x_n)$ of variables we let $|x| = n$.
Throughout $n$ is a natural number, $m,k,l$ are integers, $t,r,\lambda,\eta$ are real numbers, and $\alpha$ is an element of $\rz$.
Suppose $\alpha \in \rz$.
We let $\psi_\alpha$ denote the character $\Z \to \rz$ given by $\psi_\alpha(k) = \alpha k$.
We say that $\alpha$ is \textbf{irrational} if $\alpha = s + \Z$ for $s \in \R \setminus \Q$.
Note that $\alpha$ is irrational if and only if $\psi_\alpha$ is injective.
\newline

\noindent
All structures are first order and ``definable" means ``first-order definable, possibly with parameters".
Suppose $\Sa M$, $\Sa N$, and $\Sa O$ are structures on a common domain $M$.
Then $\Sa M$ is a \textbf{reduct} of $\Sa O$ (and $\Sa O$ is an expansion of $\Sa M$) if every $\Sa M$-definable subset of every $M^n$ is $\Sa O$-definable, $\Sa M$ and $\Sa O$ are \textbf{interdefinable} if each is a reduct of the other, $\Sa M$ is a \textbf{proper} reduct of $\Sa O$ (and $\Sa O$ is a proper expansion of $\Sa M$) if $\Sa M$ is a reduct of $\Sa O$ and $\Sa M$ is not interdefinable with $\Sa O$, and $\Sa N$ is \textbf{intermediate} between $\Sa M$ and $\Sa O$ if $\Sa M$ is a proper reduct of $\Sa N$ and $\Sa N$ is a proper reduct of $\Sa O$.
\newline

\noindent
Given a set $A$ and an injection $f : A \to M^m$ we say that the structure induced on $A$ by $\Sa M$ and $f$ is the structure on $A$ with an $n$-ary relation defining $\{ (a_1,\ldots,a_n) \in A^n : ((f(a_1)\ldots,f(a_n)) \in Y \}$ for every $\Sa M$-definable $Y \subseteq M^{nm}$.
If $A$ is a subset of $M^m$ and $f : A \to M^m$ is the identity we refer to this as the structure induced on $A$ by $\Sa M$.
\newline

\noindent
We let $\cl(X)$ denote the closure of a subset $X$ of a topological space.
\newline

\noindent
Suppose $L \subseteq L'$ are languages containing $<$, $\Sa R'$ is an $L'$-structure expanding $(\R,<)$ and $\Sa R$ is the $L$-reduct of $\Sa R'$.
The \textbf{open core} of $\Sa R'$ is the reduct of $\Sa R'$ generated by all closed $\Sa R'$-definable sets.
Furthermore $\Th(\Sa R)$ is \textbf{an open core} of $\Th(\Sa R')$ if, whenever $\Sa R' \prec \Sa N'$ then the $L'$-reduct of $\Sa N'$ is interdefinable with the open core of $\Sa N'$.
This notion clearly makes sense in much broader generality.
\newline


\noindent
We use \textbf{``semialgebraic"} as a synonym of either ``$\rfield)$-definable" or ``$\pfield)$-definable".
It will be clear in context which we mean.

\section{Model-theoretic preliminaries}
\label{section:prelim}
\noindent
Let $\Sa M$ be a structure and $\Sa M \prec \Sa N$ be highly saturated.

\subsection{Dp-minimality}
Our reference is \cite{Simon-Book}.
Recall that $\Sa M$ is \textbf{dp-minimal} if for every small set $A$ of parameters from $\Sa N$, pair $I_0,I_1$ of mutually indiscernible sequences in $\Sa N$ over $A$, and $b \in N$, $I_i$ is indiscernible over $A \cup \{b\}$ for some $i \in \{0,1\}$.
\newline

\noindent
We now describe a second definition of dp-minimality which will be useful below.
A family $( \theta_i : i \in I)$ of formulas is $n$-inconsistent if $\bigwedge_{i \in J} \theta_i$ is inconsistent for every $J \subseteq I, |J| = n$.
A pair $\varphi(x;y),\phi(x;z)$ of formulas and $n \in \N$ \textit{violate inp-minimality} if $|x| = 1$ and if for every $k \geq 1$ there are $a_1,\ldots,a_k \in M^{|y|}$ and $b_1,\ldots,b_k \in M^{|z|}$ such that $\varphi(x;a_1),\ldots,\varphi(x;a_k)$ and $\phi(x;b_1),\ldots,\phi(x;b_k)$ are both $n$-inconsistent and $\Sa M \models \exists x [\varphi(x;a_i) \land \phi(x;b_j)]$ for any $1 \leq i,j \leq k$.
We say that $\varphi(x;y)$ and $\phi(x;z)$ violate inp-minimality if there is $n$ such that $\varphi(x;y),\phi(x;z),n$ violate inp-minimality.
Then $\Sa M$ is \textbf{inp-minimal} if no pair of formulas violates inp-minimality.
Recall that $\Sa M$ is dp-minimal if and only if $\Sa M$ is inp-minimal and $\nip$.
\newline

\noindent
Fact~\ref{fact:union-reduce} is an easy application of Ramsey's theorem which we leave to the reader.

\begin{fact}
\label{fact:union-reduce}
Let $\varphi_1(x;y_1),\ldots,\varphi_m(x;y_m)$ and $\phi_1(x;z_1),\ldots,\phi_m(x;z_m)$ be formulas.
If 
\[
\varphi_\cup(x;y_1,\ldots,y_m) = \bigvee_{i = 1}^{m} \varphi_i(x;y_i), \quad \phi_\cup(x;z_1,\ldots,z_m) := \bigvee_{i = 1}^{m} \phi_i(x;z_i)
\]
violate inp-minimality then $\varphi_i(x;y_i),\phi_j(x;z_j)$ violate inp-minimality for some $i,j$.
\end{fact}

\noindent
We also leave the proof of Fact~\ref{fact:remove-finite} to the reader.

\begin{fact}
\label{fact:remove-finite}
Fix formulas $\varphi(x;y), \phi(x;y)$ with $|x| = 1$.
Suppose there is $n$ such that $\Sa M \models \forall y \exists^{\leq n} x \varphi(x;y)$.
Then $\varphi(x;y)$ and $\phi(x;y)$ do not violate inp-minimality.
\end{fact}

\subsection{External definibility}
A subset of $X$ of $M^n$ is \textbf{externally definable} if there is an $\Sa N$-definable subset $Y$ of $N^n$ such that $X = M^n \cap Y$.
By saturation the collection of externally definable sets does not depend on choice of $\Sa N$.
The \textbf{Shelah expansion} $\Sh M$ of $\Sa M$ is the expansion by all externally definable subsets of all $M^n$, equivalently, the structure induced on $M$ by $\Sa N$.
We will make frequent use of the following elementary observation.

\begin{fact}
\label{fact:external-convex}
Suppose that $\Sa M$ expands a linear order.
Then every convex subset of $M$ is externally definable.
\end{fact}

\noindent
The first claim of Fact~\ref{fact:shelah} is a theorem of Shelah~\cite{Shelah-external}, see also Chernikov and Simon~\cite{CS-I}.
The latter claims follow easily from the first, see for example Onshuus and Usvyatsov~\cite{OnUs}.

\begin{fact}
\label{fact:shelah}
If $\Sa M$ is $\nip$ then every $\Sh M$-definable subset of every $M^n$ is externally definable in $\Sa M$.
If $\Sa M$ is $\nip$ then $\Sh M$ is $\nip$, if $\Sa M$ is strongly dependent then $\Sh M$ is strongly dependent, and if $\Sa M$ is dp-minimal then $\Sh M$ is dp-minimal.
\end{fact}

\noindent
Fact~\ref{fact:cs-limit} is a theorem of Chernikov and Simon~\cite[Corollary 9]{CS-II}.

\begin{fact}
\label{fact:cs-limit}
Suppose $\Sa M$ is $\mathrm{NIP}$ and $X$ is an externally definable subset of $M^n$.
Then there is an $\Sa M$-definable family $(X_a : a  \in M^m)$ of subsets of $M^n$ such that for every finite $B \subseteq X$ there is $a \in M^m$ such that $B \subseteq X_a \subseteq X$.
\end{fact}

\noindent
Fact~\ref{fact:ms} is the Marker-Steinhorn theorem~\cite{Marker-Steinhorn}.

\begin{fact}
\label{fact:ms}
Suppose $\Sa R$ is an o-minimal expansion of $(\R,<)$.
Every externally definable subset of every $\R^n$ is definable.
Equivalently: $\Sh R$ and $\Sa R$ are interdefinable.
\end{fact}

\noindent 
Fact~\ref{fact:delon} is a theorem of Delon~\cite{Delon-def}.

\begin{fact}
\label{fact:delon}
Every subset of $\Q_p^n$ which is externally definable in $\pfield)$ is definable in $\pfield)$.
Equivalently: $\pfield)^{\mathrm{Sh}}$ and $\pfield)$ are interdefinable.
\end{fact}

\subsection{Weak minimality}
Suppose $\Sa O$ expands $\Sa M$.
We say that $\Sa O$ is $\mathbf{\Sa M}$\textbf{-minimal} if every $\Sa O$-definable subset of $M$ is definable in $\Sa M$ and we say that $\Sa O$ is \textbf{weakly $\Sa M$-minimal} if every $\Sa O$-definable subset of $M$ is externally definable in $\Sa M$.
\newline

\noindent
Suppose $L \subseteq L'$ are languages, $T'$ is a complete consistent $L'$-theory, and $T$ is the $L$-reduct of $T'$.
We say that $T'$ is $T$-\textbf{minimal} if for every $L'$-formula $\varphi(x;y), |x| = 1$ there is an $L$-formula $\phi(x;z)$ such that for every $\Sa P \models T'$ and $a \in P^{|y|}$ there is $b \in P^{|z|}$ such that $\varphi(P;a) = \phi(P;b)$.
We say that $T'$ is \textbf{weakly} $T$-\textbf{minimal} if for every $L'$-formula $\varphi(x;y), |x| = 1$ there is an $L$-formula $\phi(x;z)$ such that for every $\Sa P \models T'$, highly saturated $\Sa P \prec \Sa Q$, and $a \in P^{|y|}$, there is $b \in Q^{|z|}$ such that $\varphi(P;a) = P \cap \phi(Q;b)$.
A structure is weakly $T$-minimal if its theory is.
\newline

\noindent
Weak minimality was introduced in \cite{SW-dp}.
If $T$ is a complete theory of dense linear orders then $T'$ is $T$-minimal if and only if $T'$ is o-minimal and $T'$ is weakly $T$-minimal if and only if $T'$ is weakly o-minimal.
\newline

\noindent
Suppose $\medblackstar$ is an $\nip$-theoretic property such that $T$ has $\medblackstar$ if and only if every $T$-model omits a certain configuration involving only unary definable sets.
It is then easy to see that if $T$ is $\medblackstar$ and $T'$ is weakly $T$-minimal then $T'$ is $\medblackstar$.

\begin{fact}
\label{fact:weak-to-dp}
Suppose $T'$ is weakly $T$-minimal.
If $T$ satisfies any one of the following properties, then so does $T'$.
\begin{enumerate}
\item stability,
\item $\nip$,
\item strong dependence,
\item dp-minimality.
\end{enumerate}
\end{fact}

\section{Cyclically ordered abelian groups}
\label{section:cyclic-orders}
\noindent
We give basic definitions and results concerning cyclically ordered groups.
We also set notation to be used throughout.
See \cite{TW-cyclic} for more information and references.
\newline

\noindent
A \textbf{cyclic order} $S$ on a set $G$ is a ternary relation such that for all $a,b,c \in G$,
\begin{enumerate}
\item if $S(a,b,c)$, then $S(b,c,a)$,
\item if $S(a,b,c)$, then $\neg S(c,b,a)$,
\item if $S(a,b,c)$ and $S(a,c,d)$ then $S(a,b,d)$,
\item if $a,b,c$ are distinct, then either $S(a,b,c)$ or $S(c,b,a)$.
\end{enumerate}
An open $S$-interval is a set of the form $\{ b \in G : S(a,b,c) \}$ for some $a,c \in G$, likewise define closed and half open intervals.
A subset of $G$ is $S$-convex if it is the union of a nested family of intervals.
We drop the ``$S$" when it is clear from context.
\newline

\noindent
If $(G,+)$ is an abelian group then a cyclic group order on $(G,+)$ is a $+$-invariant cyclic order.
Suppose $S$ is a cyclic group order on $(G,+)$.
A subset of $G$ is an $\mathbf{S}$ \textbf{tmc set} if it is of the form $a + mJ$ for $S$-convex $J \subseteq G$ and $a \in G$.
We drop the ``$S$" when it is clear from context.
\newline

\noindent
Note that $\{ (a,b,c) \in G^2 : S(c,b,a) \}$ is a cyclic group order which we refer to as the \textbf{opposite} of $S$.
(If $<$ is a linear group order on $(G,+)$ then $\{ (a,b) \in G^2 : b < a \}$ is also a linear group order which we refer to as the opposite of $<$.)
\newline

\noindent
Throughout $C$ is the cyclic group order on $(\R/\Z,+)$ such that whenever $t,t',t'' \in \R$ and $0 \leq t,t',t'' < 1$ then $C(t + \Z, t' + \Z, t'' + \Z)$ holds if and only if either $t < t' < t''$, $t' < t'' < t$, or $t'' < t < t'$.
Given irrational $\alpha \in \rz$ we let $\pmb{C_\alpha}$ be the cyclic group order on $(\Z,+)$ where $C_\alpha(k,k',k'')$ if and only if $C(\alpha k,\alpha k', \alpha k'')$, so $C_\alpha$ is the pullback of $C$ by $\psi_\alpha$.
Every dense cyclic group order on $(\Z,+)$ is of this form for unique $\alpha \in \rz$.
\newline



\noindent
Let $<$ be a linear group order on $(G,+)$.
There are two associated cyclic orders:
$$ S_< := \{ (a,b,c) \in G^3 : (a < b < c) \vee (b < c < a) \vee (c < a < b) \}, $$
and
$$ S_> := \{ (a,b,c) \in G^3 : (c < b < a) \vee (b < a < c) \vee (a < c < b)\}. $$
Note that $S_<$ is the opposite of $S_>$.
See for example \cite{TW-cyclic} for a proof of Fact~\ref{fact:classify-cyclic}.

\begin{fact}
\label{fact:classify-cyclic}
Every cyclic group order on $(\Z,+)$ is either $C_\alpha$ for some irrational $\alpha \in \rz$ or $S_<$ or $S_>$ for the usual order $<$.
\end{fact}

\noindent
We will frequently apply Fact~\ref{fact:unique-iso}, which is elementary and left to the reader.

\begin{fact}
\label{fact:unique-iso}
Suppose $\bH$ is a topological group and $\gamma$ is an isomorphism $\bH \to \rz$ of topological groups.
Then $\gamma$ is unique up to sign, i.e. if $\xi : \bH \to \rz$ is a topological group isomorphism then either $\xi = \gamma$ or $\xi = - \gamma$
.
\end{fact}

\subsection{The universal cover}
We describe the universal cover of $(G,+,S)$.
A universal cover of $(G,+,S)$ is an ordered abelian group $(H,+,<)$, a distinguished positive $u \in H$ such that $u\Z$ is cofinal in $H$, and a surjective group homomorphism $\pi : H \to G$ with kernel $u\Z$ such that if $a,b,c \in H$ and $0 \leq a,b,c < u$ then $S(\pi(a), \pi(b),\pi(c) )$ if and only if we either have $a < b < c$, $b < c < a$, or $c < a < b$.
The universal cover $(H,+,<,u,\pi)$ is unique up to unique isomorphism and every cyclically ordered abelian group has a universal cover.
\newline

\noindent
So $(\R,+,<,1,\pi)$ is a universal cover of $\rzc$, where $\pi(t) = t + \Z$ for all $t \in \R$ and $(\Z + s\Z,+,<,1,\pi)$ is a universal cover of $\zca$ when $\alpha = s + \Z$.

\section{Definable groups}
\label{section:def-groups}
\noindent
We recall some basic facts from the extensive theory of definable groups in o-minimal structures.
\textbf{Throughout this section $\Sa R$ is an o-minimal expansion of $\rfield)$, $\bH$ is an $\Sa R$-definable group, and ``dimension" without modification is the o-minimal dimension.}
\newline

\noindent
Fact~\ref{fact:pillay-top} follows from work of Pillay~\cite{pillay-groups-fields} and \cite[10.1.8]{Lou}.

\begin{fact}
\label{fact:pillay-top}
There is an $\Sa R$-definable group $\G$ with underlying set $G \subseteq \R^m$ such that $\G$ is a topological group with respect to the topology induced by $\R^m$ and an $\Sa R$-definable group isomorphism $\xi : \bH \to \G$.
If $\G'$ is an $\Sa R$-definable group with underlying set $G' \subseteq \R^n$, $\G'$ is a topological group with respect to the topology induced by $\R^n$, and $\xi' : \bH \to \G'$ is an $\Sa R$-definable group isomorphism, then $\xi' \circ \xi^{-1}$ is a topological group isomorphism $\G \to \G'$.
\end{fact}

\noindent
We let $\Cal T_\bH$ be the canonical group topology on $\bH$ and consider $\bH$ as a topological group.
Recall that any connected topological group of topological dimension one is isomorphic (as a topological group) to either $(\R,+)$ or $(\rz,+)$.
It follows that if $\bH$ is one-dimensional and connected then $\bH$ is isomorphic as a topological group to either $(\R,+)$ or $(\rz,+)$.
In the first case we say that $\bH$ is a \textbf{line group}, in the second case $\bH$ is a \textbf{circle group}.
\newline

\noindent
Suppose $X$ is an $\Sa R$-definable subset of $\R^m$.
An easy application of the good directions lemma \cite[Theorem 4.2]{Lou} shows that if $X$ is homeomorphic to $\R$ then there is an $\Sa R$-definable homeomorphism $X \to \R$ and if $X$ is homeomorphic to $\rz$ then there is an $\Sa R$-definable homeomorphism from $X$ to the unit circle.
(The analogous fact fails in higher dimensions, there are homeomorphic semialgebraic sets $X,X'$ for which there is no homeomorphism $X \to X'$ definable in an o-minimal expansion of $\rfield)$, this is a consequence of Shiota's o-minimal Hauptvermutung~\cite{shiota-haup} together with the failure of the classical Hauptvermutung.)
Fact~\ref{fact:get-cyclic} easily follows.

\begin{fact}
\label{fact:get-cyclic}
Suppose $\bH$ is one-dimensional, connected, and has underlying set $H$ and group operation $\oplus$.
Then there is a unique up to opposite $\Sa R$-definable cyclic group order $S$ on $\bH$.
If $\bH$ is a line group then $(H,+,S)$ is isomorphic to $(\R,+,S_<)$.
If $\bH$ is a circle group the $(H,+,S)$ is isomorphic to $\rzc$.
\end{fact}

\noindent
So if $\bH$ is one-dimensional and connected and $A$ is a subgroup of $\bH$ then we may speak without ambiguity of a tmc subset of $A$.
\newline

\noindent
Finally we recall the interpretation-rigidity theorem for o-minimal expansions of $\rfield)$.
Fact~\ref{fact:opp} is due to Otero, Peterzil, and Pillay~\cite{OPP-groups-rings}.

\begin{fact}
\label{fact:opp}
Let $\F$ be an infinite field interpretable in $\Sa R$.
Then there is either an $\Sa R$-definable field isomorphism $\F \to \rfield)$ or $\F \to (\C,+,\times)$.
It follows that if an expansion $\Sa S$ of $\rfield)$ is interpretable in $\Sa R$ then $\Sa S$ is isomorphic to a reduct of $\Sa R$, and if a structure $\Sa M$ is mutually interpretable with $\Sa R$ then $\Sa R$ is (up to interdefinibility) the unique expansion of $\rfield)$ mutually interpretable with $\Sa M$.
\end{fact}

\section{What we know about dp-minimal expansions of $(\Z,+)$}
\label{section:survey}
\noindent
We survey what is known about dp-minimal expansions of $(\Z,+)$.
\newline

\noindent
The first result on dp-minimal expansions of $(\Z,+)$ is Fact~\ref{fact:toomany}, proven in \cite[Proposition 6.6]{toomanyI}.
Fact~\ref{fact:toomany} follows easily from two results, the Michaux-Villemaire theorem~\cite{MV} that there are no proper $(\Z,+,<)$-minimal expansions of $(\Z,+,<)$, and Simon's theorem~\cite[Lemma 2.9]{Simon-dp} that a definable family of unary sets in a dp-minimal expansion of a linear order has only finitely many germs at infinity.

\begin{fact}
\label{fact:toomany}
There are no proper dp-minimal expansions of $(\Z,+,<)$.
Equivalently: there are no proper dp-minimal expansions of $(\N,+)$.
\end{fact}

\noindent
The authors of \cite{toomanyI} raised the question of whether there is a dp-minimal expansion of $(\Z,+)$ which is not a reduct of $(\Z,+,<)$.
Conant and Pillay~\cite{CoPi} proved Fact~\ref{fact:cp}.
Their proof relies on earlier work of Palac\'{i}n and Sklinos~\cite{PS-superstable}, who apply the Buechler dichotomy theorem and other sophisticated tools of stability theory.

\begin{fact}
\label{fact:cp}
There are no proper stable dp-minimal expansions of $(\Z,+)$.
\end{fact}

\noindent
Conant~\cite{conant} proved Fact~\ref{fact:conant} via a geometric analysis of $(\Z,+,<)$-definable sets.
Facts~\ref{fact:cp} and \ref{fact:weak-to-dp} show that a proper dp-minimal expansion of $(\Z,+)$ is not $\Th(\Z,+)$-minimal.
Alouf and d'Elb\'{e}e~\cite{AldE} used this to give a quicker proof of Fact~\ref{fact:conant}.

\begin{fact}
\label{fact:conant}
There are no intermediate structures between $(\Z,+)$ and $(\Z,+,<)$.
\end{fact}

\noindent
Alouf and d'Elb\'{e}e~\cite{AldE} proved Fact~\ref{fact:ae}.
Given a prime $p$ we let $\valp$ be the $p$-adic valuation on $(\Z,+)$ and $\prec_p$ be the partial order on $\Z$ where $m \prec_p n$ if and only if $\valp(m) < \valp(n)$.
We can view $\zpp$ as either $(\Z,+,\prec_p)$ or as the two sorted structure with disjoint sorts $\Z$ and $\N \cup \{\infty\}$, addition on $\Z$, and $\valp : \Z \to \N \cup \{\infty\}$.
It makes no difference which of these two options we take.

\begin{fact}
\label{fact:ae}
Let $p$ be a prime.
Then $\zpp$ is dp-minimal and $(\Z,+)$-minimal, and there are no structures intermediate between $(\Z,+)$ and $\zpp$.
\end{fact}

\noindent
Alouf and d'Elb\'{e}e also show that $(\Z,+,(\valp)_{p \in I})$ has dp-rank $|I|$ for any nonempty set $I$ of primes.
So if $p \neq q$ are primes then $\zpp$ and $(\Z,+,\mathrm{Val}_q)$ do not have a common dp-minimal expansion.
\newline

\noindent
So far we have described countably many dp-minimal expansions of $(\Z,+)$.
Fact~\ref{fact:basic-cyclic}, proven by Tran and Walsberg~\cite{TW-cyclic}, shows that there is an uncountable collection of dp-minimal expansions of $(\Z,+)$, no two of which are interdefinable.

\begin{fact}
\label{fact:basic-cyclic}
Suppose $\alpha,\beta \in \rz$ are irrational.
Then $\zca$ is dp-minimal.
Furthermore $\zca$ and $\zcb$ are interdefinable if and only if $\alpha$ and $\beta$ are $\Z$-linearly dependent.
\end{fact}

\noindent
Fact~\ref{fact:basic-cyclic}, Fact~\ref{fact:classify-cyclic}, and dp-minimality of $(\Z,+,<)$ together show that any expansion of $(\Z,+)$ by a cyclic group order is dp-minimal.
\newline


\noindent
It is shown in \cite{TW-cyclic} that every unary definable set in every elementary extension of $\zca$ is a finite union of tmc sets.
It follows by Fact~\ref{fact:weak-to-dp} that if $\Sa Z$ expands $\zca$ and every unary definable set in every elementary extension of $\Sa Z$ is a finite union of tmc sets, then $\Sa Z$ is dp-minimal.
A converse is proven in \cite{SW-dp}.

\begin{fact}
\label{fact:simon-w}
Fix irrational $\alpha \in \rz$.
Suppose $\Sa Z$ is a dp-minimal expansion of $\zca$.
Then $\Sa Z$ is weakly $\Th\zca$-minimal (equivalently: every unary definable set in every elementary extension of $\Sa Z$ is a finite union of tmc sets).
\end{fact}

\noindent
In particular a dp-minimal expansion of $\shc$ cannot add new unary sets.
\newline

\noindent
Suppose $\alpha,\beta \in \rz$ are irrational and $\Z$-linearly independent.
An easy application of Kronecker density shows that if $I$ is an infinite and co-infinite $C_\alpha$-convex set then $I$ is not a finite union of $C_\beta$-tmc sets, see \cite{TW-cyclic}.
Fact~\ref{fact:tmc} follows.

\begin{fact}
\label{fact:tmc}
Suppose $\alpha,\beta \in \rz$ are irrational and $\Z$-linearly independent.
Suppose $\Sa Z_\alpha$ is a dp-minimal expansion of $\zca$ and $\Sa Z_\beta$ is a dp-minimal expansion of $(\Z,+,C_\beta)$.
If $I$ is an infinite and co-infinite $C_\alpha$-interval then $I$ is not $\Sa Z_\beta$-definable, and vice versa.
So $\Sa Z_\alpha$ defines a subset of $\Z$ which is not $\Sa Z_\beta$-definable, and vice versa.
In particular $\Sa Z_\alpha$ and $\Sa Z_\beta$ do not have a common dp-minimal expansion.
\end{fact}

\noindent
We now describe a striking recent result of Alouf~\cite{aloug}.
We first recall Fact~\ref{fact:exists-0}, a special case of \cite[Lemma 3.1]{JaSiWa}.

\begin{fact}
\label{fact:exists-0}
Suppose $\Sa G$ is a dp-minimal expansion of a group $G$ which defines a non-discrete Hausdorff group topology on $G$.
Then $\Sa G$ eliminates $\exists^\infty$.
\end{fact}

\noindent
Fact~\ref{fact:exists-0} shows that any dp-minimal expansion of $\zpp$ or $(\Z,+,C_\alpha)$ eliminates $\exists^\infty$.
Fact~\ref{fact:exists-1} is proven in~\cite{aloug}.

\begin{fact}
\label{fact:exists-1}
Suppose $\Sa Z$ is a dp-minimal expansion of $(\Z,+)$ which either
\begin{enumerate}
\item does not eliminate $\exists^\infty$,
\item or defines an infinite subset of $\N$.
\end{enumerate}
Then $\Sa Z$ defines $<$.
\end{fact}

\noindent
So $(\Z,+,<)$ is, up to interdefinibility, the only dp-minimal expansion of $(\Z,+)$ which does not eliminate $\exists^\infty$.
Conjecture~\ref{conj:top} is now natural.

\begin{conj}
\label{conj:top}
Any proper dp-minimal expansion of $(\Z,+)$ which eliminates $\exists^\infty$ defines a non-discrete group topology on $(\Z,+)$.
\end{conj}

\noindent
Johnson~\cite{Johnson-top} shows that a dp-minimal expansion of a field which is not strongly minimal admits a definable non-discrete field topology.
His proof makes crucial use of the fact that any dp-minimal expansion of a field eliminates $\exists^\infty$.

\subsection{Interpretations}
We describe what we know about interpretations between dp-minimal expansions of $(\Z,+)$.
We suspect that bi-interpretable dp-minimal expansions of $(\Z,+)$ are interdefinable.



\begin{prop}
\label{prop:eq}
Fix irrational $\alpha \in \rz$.
Suppose $\Sa Z$ is a dp-minimal expansion of $\zca$.
Then $\Sa Z^\mathrm{eq}$ eliminates $\exists^\infty$, so $\Sa Z$ does not interpret $(\Z,+,<)$ or $\zpp$ for any prime $p$.
\end{prop}

\noindent
Note that $\zpp^{\mathrm{eq}}$ does not eliminate $\exists^\infty$ as $\zpp$ interprets $(\N,<)$.
\newline

\noindent
Given a structure $\Sa M$ we say that $\Sa M^{\mathrm{eq}}$ eliminates $\exists^\infty$ in one variable if for every definable family $(E_a : a \in M^k)$ of equivalence relations on $M$ there is $n$ such that for all $a \in M^k$ we either have $|M/E_a| < n$ or $|M/E_a| \geq \aleph_0$.
Proposition~\ref{prop:eq} requires Fact~\ref{fact:eq-reduce}, which is routine and left to the reader.

\begin{fact}
\label{fact:eq-reduce}
Let $\Sa M \prec \Sa N$ be highly saturated.
Suppose that $\Sa N$ eliminates $\exists^\infty$ and there is no $\Sa N$-definable equivalence relation on $N$ with infinitely many infinite classes.
Then $\Eq M$ eliminates $\exists^\infty$.
\end{fact}

\noindent
We now prove Proposition~\ref{prop:eq}.
We use the notation and results of \cite{SW-dp}, so the reader will need to have a copy of that paper at hand.

\begin{proof}
Let $(H,+,<,u,\pi)$ be a universal cover of $\zca$, $I : = (-u,u)$.
So let $\Sa I$ be the structure induced on $I$ by $\Sa Z$ and $\pi$.
It is shown in \cite{SW-dp} that $\Sa I$ and $\Sa Z$ define isomorphic copies of each other, so it suffices to show that $\Sa I^{\mathrm{eq}}$ eliminates $\exists^\infty$.
Let $\Sa I \prec \Sa J$ be highly saturated.
The proof of Fact~\ref{fact:exists-0} shows that $\Sa J$ eliminates $\exists^\infty$.
We show that every $\Sa J$-definable equivalence relation on $J$ has only finitely many infinite classes and apply Fact~\ref{fact:eq-reduce}.
\newline

\noindent
Suppose $E$ is a $\Sa J$-definable equivalence relation on $J$ with infinitely many infinite classes.
By \cite[Lemma 8.7]{SW-dp} there is a finite partition $\Cal A$ of $J$ into $\Sa J$-definable sets such that every $E$-class is a finite union of sets of the form $K \cap A$ for convex $K$ and $A \in \Cal A$.
Fix $A \in \Cal A$ which intersects infinitely many $E$-classes.
Note that the intersection of each $E$-class with $A$ is a finite union of convex sets.
Let $F$ be the equivalence relation on $J$ where $a < b$ are $F$-equivalent if and only if there are $a' < a < b < b'$ such that $a',b' \in A$, $a'$ and $b'$ are $E$-equivalent, and $a',b'$ lie in the same convex component of $E_{a'} \cap A$.
It is easy to see that every $F$-class is convex and there are infinitely many $F$-classes.
However, it is shown in the proof of \cite[Lemma 8.7]{SW-dp} that any definable equivalence relation on $J$ with convex equivalence classes has only finitely many infinite classes.
\end{proof}

\noindent
Fact~\ref{fact:ntp} is proven in \cite[Proposition 5.6]{big-nip}.

\begin{fact}
\label{fact:ntp}
Suppose $\Sa Z$ is an $\mathrm{NTP}_2$ expansion of $(\Z,<)$ and $\Sa G$ is an expansion of a group $G$ which defines a non-discrete Hausdorff group topology on $G$.
Then $\Sa Z$ does not interpret $\Sa G$.
So in particular an $\mathrm{NTP}_2$ expansion of $(\Z,+,<)$ does not interpret  $\zca$ for any irrational $\alpha \in \rz$ or $\zpp$ for any prime $p$.
\end{fact}

\noindent
In Section~\ref{section:p-adic} we construct a dp-minimal expansion $\Sa P$ of $\zpp$ which defines addition on the value set, so in particular $\Sa P$ interprets $(\Z,+,<)$.

\section{New dp-minimal expansions of $\zca$}
\label{section:examples}

\noindent
We describe new dp-minimal expansions of $\zca$.

\subsection{Dense pairs}
We first recall Hieronymi and G\"{u}naydin~\cite{HiGu}.
Let $\bH$ be an abelian semialgebraic group with underlying set $H \subseteq \R^m$ and group operation $\oplus$, and $A$ be a subgroup of $\bH$.
Then $A$ has the \textbf{Mordell-Lang property} if for every $f \in \R[x_1,\ldots,x_{nm}]$ the set
$ \{ a \in A^n : f(a) = 0 \} $ is a finite union of sets of the form
$$\{ (a_1,\ldots,a_n) \in A^n : k_1a_1 \oplus \ldots \oplus k_na_n = b \} \quad \text{for some  } k_1,\ldots,k_n \in \Z, b \in A.$$
We say that $\bH$ is a \textbf{Mordell-Lang} group if every finite rank subgroup of $\bH$ has the Mordell-Lang property.
Fact~\ref{fact:GH} is essentially in \cite{HiGu}, but see the comments below.

\begin{fact}
\label{fact:GH}
Suppose $\bH$ is a one-dimensional connected Mordell-Lang group.
Let $A$ be a dense finite rank subgroup of $\bH$.
Then $\rfield,A)$ is $\nip$, $\Th\rfield)$ is an open core of $\Th\rfield,A)$, and every subset of $A^k$ definable in $\rfield,A)$ is a finite union of sets of the form $b \oplus n(X \cap A^k)$ for semialgebraic $X$ and $b \in A^k$.
\end{fact}

\noindent
Note that the last claim of Fact~\ref{fact:GH} shows that structure induced on $A$ by $\rfield)$ is interdefinable with the structure induced by $\rfield,A)$ are interdefinable.
\newline

\noindent
The reader will not find the exact statement of the last claim of Fact~\ref{fact:GH} in \cite{HiGu}.
It is incorrectly claimed in \cite[Proposition 3.10]{HiGu} that every subset of $A^k$ definable in $\rfield,A)$ is a finite union of sets of the form $X \cap (b \oplus nA^k)$ where $X \subseteq \bH$ is semialgebraic.
This is true when $\bH$ is a line group, but fails when $\bH$ is a circle group.
If $\bH$ is a circle group and $I$ is an infinite and co-infinite open interval in $\bH$ then $2I$ is not of this form.
A slightly corrected version of the proof of \cite[Proposition 3.10]{HiGu} yields the last statement of Fact~\ref{fact:GH}\footnote{Thanks to Philipp Hieronymi for discussions on this point.}.
\newline

\noindent
Proposition~\ref{prop:dp-cyclic} is proven in \cite{SW-dp}.

\begin{prop}
\label{prop:dp-cyclic}
Suppose $(G,+,S)$ is a cyclically order abelian group and $\Sa G$ expands $(G,+,S)$.
Suppose $|G/nG| < \aleph_0$ for all $n$.
Then $\Sa G$ is dp-minimal if and only if every unary definable set in every elementary extension of $\Sa G$ is a finite union of tmc sets.
So $\Sa G$ is dp-minimal if and only if $\Th(\Sa G)$ is weakly $\Th(G,+,S)$-minimal.
\end{prop}

\noindent
Let $\Sa A$ be the structure induced on $A$ by $\rfield)$.
Fact~\ref{fact:GH} shows that every $\Sa A$-definable unary set is a finite union of tmc sets, and that the same claim holds in every elementary extension of $\Sa A$.
Proposition~\ref{prop:GH-1}  follows.

\begin{prop}
\label{prop:GH-1}
If $\bH$ is a one-dimensional connected Mordell-Lang group and $A$ is a dense finite rank subgroup of $\bH$, then the structure induced on $A$ by $\rfield)$ is dp-minimal.
So if $\bH$ is a Mordell-Lang circle group and $\chi : \Z \to \bH$ is an injective character then the structure induced on $\Z$ by $\rfield)$ and $\chi$ is dp-minimal.
\end{prop}

\noindent
Of course Proposition~\ref{prop:GH-1} is only relevant because there are semialgebraic Mordell-Lang circle groups by the general Mordell-Lang conjecture.
This is a theorem of Faltings, Vojta, McQuillan and others, see \cite{mazur-survey} for a survey.

\begin{fact}
\label{fact:mordell-lang}
If $W$ is a semiabelian variety defined over $\C$, $V$ is a subvariety of $W$, and $\Gamma$ is a finite rank subgroup of $W(\C)$, then $\Gamma \cap V(\C)$ is a finite union of cosets of subgroups of $\Gamma$.
So $(\R_{>0},\times)$, the unit circle equipped with complex multiplication, and the real points of an elliptic curve defined over $\R$ are all Mordell-Lang groups.
\end{fact}

\subsection{Specific examples}
Suppose $\bH$ is a semialgebraic group equipped with $\Cal T_\bH$.
By \cite{HrushovskiPillay} there is an open neighbourhood $U \subseteq \bH$ of the identity, an algebraic group $W$ defined over $\R$, a neighbourhood $V \subseteq W(\R)$ of the identity, and a semialgebraic local group isomorphism $U \to V$.
We say that $\bH$ is semiabelian when $W$ is semiabelian.
Suppose $\bH$ is one-dimensional.
Then $W$ is one dimensional, so we may take $W(\R)$ to be either $(\R,+)$, $(\R^\times,\times)$, the unit circle, or the real points of an elliptic curve.
In the latter three cases $\bH$ is semiabelian.
\newline

\noindent
One-dimensional semialgebraic groups were classified up to semialgebraic isomorphism by Madden and Stanton~\cite{nash-group}.
There are three families of semiabelian semialgebraic circle groups.
\newline

\noindent
We describe the first family.
Given $\lambda > 1$ we let $\G_\lambda := ([1,\lambda),\otimes_\lambda)$ where $t \otimes_\lambda t' = tt'$ when $tt' < \lambda$ and $t \otimes_\lambda t' = tt'\lambda^{-1}$ otherwise.
Let $\lambda,\eta > 1$.
The unique (up to sign) topological group isomorphism $\G_\lambda \to \G_\eta$ is $t \mapsto t^{\log_\lambda \eta}$.
So $\G_\lambda$ and $\G_\eta$ are semialgebraically isomorphic if and only if $\log_\lambda \eta \in \Q$.

\begin{lem}
\label{lem:mult-ml}
Fix $\lambda > 1$.
Suppose $A$ is a finite rank subgroup of $\G_\lambda$.
Then $\rfield,A)$ is $\nip$, $\Th\rfield)$ is an open core of $\Th\rfield,A)$, and the structure induced on $A$ by $\rfield,A)$ is dp-minimal.
\end{lem}

\noindent
We let $S$ be the cyclic order on $[1,\lambda)$ where $S(t,t',t'')$ if and only if either $t < t' < t''$, $t' < t'' < t$, or $t'' < t < t'$.
So $S$ is the unique (up to opposite) semialgebraic cyclic group order on $\G_\lambda$.

\begin{proof}
Identify $\G_\lambda$ with $(\R_{>0}/\lambda^{\Z},\times)$ and let $\rho$ be the quotient map $\R_{>0} \to \G_\lambda$.
So $(\R_{>0},\times,<,\lambda,\rho)$ is a universal cover of $(\G_\lambda,S)$.
Let $H := \rho^{-1}(A)$.
So $H$ is finite rank and $(H,\times,<,\lambda,\rho)$ is a universal cover of $(A,\otimes_\lambda,S)$.
As $(\R_{>0},\times)$ is a Mordell-Lang group and $H$ is dense in $\R_{>0}$, $\rfield,H)$ is $\nip$, $\Th\rfield)$ is an open core of $\Th\rfield,H)$, and the structure induced on $H$ by $\rfield,H)$ is dp-minimal.
Observe that $A$ is definable in $\rfield,H)$.
So $\rfield,A)$ is $\nip$ and $\Th\rfield)$ is an open core of $\Th\rfield,A)$.
Finally the structure induced on $A$ by $\rfield)$ is interdefinable with the structure induced on $H \cap [0,\lambda)$ by $\rfield)$.
So the structure induced on $A$ by $\rfield)$ is dp-minimal.
\end{proof}

\noindent
The unique (up to sign) topological group isomorphism $\gamma : \rz \to \G_\lambda$ is $\gamma(t + \Z) = \lambda^{t - \lfloor t \rfloor}$.
Given irrational $\alpha = s + \Z \in \rz$ we let $\chi_\alpha : \Z \to \G_\lambda$ be
$$\chi_\alpha(k) := \gamma(\alpha k) = \lambda^{sk - \lfloor sk \rfloor}$$
and let $\Sa G_{\alpha,\lambda}$ be the structure induced on $\Z$ by $\rfield)$ and $\chi_\alpha$.

\begin{prop}
\label{prop:lamb}
Let $\alpha \in \rz$ be irrational and $\lambda > 1$.
Then $\Sa G_{\alpha,\lambda}$ is a dp-minimal expansion of $\zca$.
\end{prop}

\noindent
Let $\bS$ be the unit circle equipped with complex multiplication.
The second family of consists of $\bS$ and other circle groups constructed from $\bS$ in roughly the same way that $\G_\lambda$ is constructed from $(\R_{>0},\times)$.
We only discuss $\bS$.
The unique (up to sign) topological group isomorphism $\gamma : \rz \to \bS$ is given by $\gamma(t + \Z) = e^{2\pi it}$.
Given irrational $\alpha = s + \Z \in \rz$ we let $\chi_\alpha : \Z \to \bS$ be
$$\chi_\alpha(k) := \gamma(\alpha k) = e^{2\pi i s k}.$$
and let $\Sa S_\alpha$ be the structure induced on $\Z$ by $\rfield)$ and $\chi_\alpha$.


\begin{prop}
\label{prop:gh-circle}
Let $\alpha \in \rz$ be irrational.
Then $\Sa S_\alpha$ is a dp-minimal expansion of $\zca$.
\end{prop}

\noindent
The third family comes from elliptic curves.
Given an elliptic curve $\E$ defined over $\R$ we let $\E(\R)$ be the real points of $\E$.
We consider $\E$ as a subvariety of $\mathbb{P}^2$ via the Weierstrass embedding.
We let $\E^0(\R)$ be the connected component of the identity of $\E(\R)$, so $\E^0(\R)$ is a semialgebraic circle group.
The fourth family of semialgebraic circle groups consists of such $\E^0(\R)$ and circle groups constructed from $\E^0(\R)$ in roughly the same way as $\G_\lambda$ is constructed from $(\R_{>0},\times)$.
We only discuss $\E_0(\R)$.
\newline

\noindent
Fix $\lambda > 0$ and let $\Lambda$ be the lattice $\Z + i\lambda\Z$.
Let $\E_\lambda$ be the elliptic curve associated to $\Lambda$, recall that $\E_\lambda$ is defined over $\R$ and any elliptic curve defined over $\R$ is isomorphic to some $\E_\lambda$.
Given $\eta > 0$ there is a semialgebraic group isomorphism $\E^0_\lambda(\R) \to \E^0_\eta(\R)$ if and only if $\lambda/\eta \in \Q$, see \cite{nash-group}.
\newline

\noindent
Let $\wp_\lambda$ be the Weierstrass elliptic function associated to $\Lambda$ and $\mathfrak{p}_\lambda : \R \to \E^0_\lambda(\R)$ be given by $\frp(t) = [\wp_\lambda(t) : \wp'_\lambda(t) : 1 ]$.
The unique (up to sign) topological group isomorphism $\gamma : \rz \to \E^0_\lambda(\R)$ is $\gamma(t + \Z) = \frp(t)$.
Fix irrational $\alpha = s + \Z \in \rz$ and let $\chi_\alpha : \Z \to \E^0_\lambda(\R)$ be the character 
$$\chi_\alpha(k) := \gamma(\alpha k) = \frp(s k) = [\wp_\lambda(sk) : \wp'_\lambda(sk) : 1 ].$$
Let $\Sa E_{\alpha,\lambda}$ be the structure induced on $\Z$ by $\rfield)$ and $\chi_\alpha$.

\begin{prop}
\label{prop:elliptic-1}
Let $\alpha \in \rz$ be irrational and $\lambda > 0$.
Then $\Sa E_{\alpha,\lambda}$ is a dp-minimal expansion of $\zca$.
\end{prop}

\subsection{Another possible family of expansions}
\label{section:another}
We describe an approach to constructing uncountably many dp-minimal expansions of each example described above.
Let $I$ be a closed bounded interval with interior.
Let $C^\infty(I)$ be the topological vector space of smooth functions $I \to \R$ where the topology is that induced by the seminorms $f \mapsto \max 
\{ |f^{(n)}(t)| : t \in I \}$.
So $C^\infty(I)$ is a Polish space.
Le Gal has shown that the set of $f \in C^\infty(I)$ such that $\rfield,f)$ is o-minimal is comeager~\cite{LGal}.

\begin{conj}
\label{conj:le-gal}
Let $\bH$ be a semialgebraic Mordell-Lang circle group, $\gamma$ be the unique (up to sign) topological group isomorphism $\rz \to \bH$, $\alpha \in \rz$ be irrational, $\chi : \Z \to \bH$ be given by $\chi(k) = \gamma(\alpha k)$, and $A := \chi(\Z)$.
There is a comeager subset $\Lambda$ of $C^\infty(I)$ (possibly depending on $\alpha$) such that if $f \in \Lambda$ then
\begin{enumerate}
    \item $\rfield,f)$ is o-minimal,
    \item if $f \neq g$ are in $\Lambda$ then $\rfield,f)$ and $\rfield,g)$ are not interdefinable.
    \item Every $\rfield,f)$-definable group is definably isomorphic to a semialgebraic group and any $\rfield,f)$-definable homomorphism between semialgebraic groups is semialgebraic.
    \item $\rfield,A)$ is $\nip$ and $\Th\rfield)$ is an open core of $\Th\rfield,A)$.
    \item Every $\rfield,A)$-definable subset of $A^k$ is a finite union of sets of the form $b \oplus n(X \cap A^k)$ for semialgebraic $X$ and $b \in A^k$.
    So in particular the structure induced on $A$ by $\rfield)$ is a dp-minimal expansion of $\zca$.
\end{enumerate}
\end{conj}

\noindent
Gorman, Hieronymi, and Kaplan generalized the Mordell-Lang property to an abstract model theoretic setting~\cite{GoHi-Pairs}.
Item $(4)$ of Conjecture~\ref{conj:le-gal} should follow by verifying that the conditions in their paper are satisfied.
\newline

\noindent
Suppose Conjecture~\ref{conj:le-gal} holds.
Let $\Sa H_\alpha$ be the structure induced on $\Z$ by $\rfield)$ and $\chi$ and for each $f \in \Lambda$ let $\Sa H_{\alpha,f}$ be the structure induced on $\Z$ by $\rfield,f)$ and $\chi$.
So each $\Sa H_{\alpha,f}$ is a dp-minimal expansion of $\Sa H_\alpha$.
\newline

\noindent
It is easy to see that our expansions of $\zca$ define the same subsets of $\Z$ as $\shc$, so is $\shc$ a reduct of these expansions?
It is intuitively obvious that these expansions defines subsets of $\Z^2$ which are not definable in $\shc$, but how do we show this?
When are two of the expansions described above interdefinable?
We now develop tools to answer these questions.

\section{The o-minimal completion}
\label{section:completion}
\noindent
We associate an o-minimal expansion of $\rzc$ to a strongly dependent expansion of $\zca$.
We will show that $\shc$ is interdefinable with the structure induced on $\Z$ by $\rzc$ and $\psi_\alpha$.
It will follow that each of the dp-minimal expansions of $\zca$ describe above in fact expands $\shc$.
\newline

\noindent
We first recall the completion  of an $\nip$ expansion of a dense archimedean ordered abelian group defined in \cite{big-nip}.

\subsection{The linearly ordered case}
Suppose that $(H,+,<)$ is a dense subgroup of $(\R,+,<)$, $\Sa H$ is an expansion of $(H,+,<)$, and $\Sa H \prec \Sa N$ is highly saturated.
Let $\mfin$ be the convex hull of $H$ in $N$ and $\minf$ be the set of $a \in N$ such that $|a| < b$ for all positive $b \in H$.
We identify $\mfin/\minf$ with $\R$ so the quotient map $\st : \mfin \to \R$ is the usual standard part map.
Note that $\mfin$ and $\minf$ are both $\Sh H$-definable so we regard $\R$ as an imaginary sort of $\Sh N$.
We let $\st : \mfin^n \to \R^n$ be given by
$ \st(a_1,\ldots,a_n) = (\st(a_1),\ldots,\st(a_n)). $
Fact~\ref{fact:complete-0} is \cite[Theorem F]{big-nip}.

\begin{fact}
\label{fact:complete-0}
Suppose $\Sa H$ is $\nip$.
Then the following structures are interdefinable.
\begin{enumerate}
\item The structure $\Sq H$ on $\R$ with an $n$-ary relation symbol defining the closure in $\R^n$ of every subset of $H^n$ which is externally definable in $\Sa H$.
\item The structure on $\R$ with an $n$-ary relation symbol defining, for  each $\Sa N$-definable subset $X$ of $N^n$, the image of $\mfin^n \cap X$ under the standard part map $\mfin^n \to \R^n$.
\item The open core of the structure induced on $\R$ by $\Sh N$.
\end{enumerate}
Furthermore the structure induced on $H$ by $\Sq H$ is a reduct of $\Sh H$.
If $\Sa H$ is strongly dependent then $\Sq H$ is interdefinable with the structure induced on $\R$ by $\Sh N$.
\end{fact}

\noindent
The completion $\Sq H$ should be ``at least as tame" as $\Sa H$ because $\Sq H$ is interpretable in $\Sh N$.
In general $\Sh H$ is not interdefinable with the structure induced on $H$ by $\Sq H$.
Suppose $H = \R$ and $\Sa H = (\R,+,<,\Q)$, it follows from Theorem~\ref{thm:complete-induced} and the quantifier elimination for $(\R,+,<,\Q)$ that $(\R,+,<,\Q)^\square$ is interdefinable with $(\R,+,<)$.
Recall that $(\R,+,<,\Q)$ has dp-rank two~\cite{DoGo}.
We expect that if $\Sa H$ is dp-minimal then $\Sh H$ is interdefinable with the structure induced on $H$ by $\Sq H$.
Note that if $H = \R$ and $\Sa H$ is dp-minimal then $\Sa H$ is o-minimal by \cite{Simon-dp}, so by the Marker-Steinhorn theorem $\Sq H$ is the open core of $\Sa H$, so $\Sq H$ and $\Sa H$ are interdefinable as any o-minimal stucture is interdefinable with its open core.

\subsection{The cyclically ordered case}
We only work over $\zca$, but everything goes through for a cyclic order on an abelian group induced by an injective character to $\rz$.
Fix irrational $\alpha \in \rz$.
Abusing notation we let $\psi_\alpha : \Z^n \to (\R/\Z)^n$ be given by $\psi_\alpha(k_1,\ldots,k_n) = (\alpha k_1,\ldots,\alpha k_n)$.
If $\beta \in \rz$ is irrational then $C_\alpha = C_\beta$ if and only if $\alpha = \beta$, so we can recover $\psi_\alpha$ from $\zca$.
\newline

\noindent
Let $\Sa Z \prec \Sa N$ be highly saturated.
We define a standard part map $\st : N \to \rz$ by declaring $\st(a)$ to be the unique element of $\rz$ such that for all integers $k,k'$ we have $C(\alpha k,\st(a),\alpha k')$ if and only if $C_\alpha(k,a,k')$.
Note that $\st$ is a homomorphism and let $\minf$ be the kernal of $\st$.
We identify $N/\minf$ with $\rz$ and $\st$ with the quotient map.
Note that $\minf$ is convex, hence $\Sh N$-definable.
So we consider $\rz$ to be an imaginary sort of $\Sh N$.

\begin{prop}
\label{prop:complete-1}
Suppose $\Sa Z$ is $\nip$.
The following structures are interdefinable.
\begin{enumerate}
\item The structure $\Sq Z$ on $\rz$ with an $n$-ary relation symbol defining the closure in $(\rz)^n$ of $\psi_\alpha(X)$ for every $X \subseteq \Z^n$ which is externally definable in $\Sa Z$.
\item The structure on $\rz$ with an $n$-ary relation symbol defining the image  of each $\Sa N$-definable $X \subseteq N^n$ under the standard part map $N^n \to (\rz)^n$.
\item The open core of the structure induced on $\rz$ by $\Sh N$.
\end{enumerate}
Furthermore the structure induced on $\Z$ by $\Sq Z$ and $\psi_\alpha$ is a reduct of $\Sh Z$.
If $\Sa Z$ is strongly dependent then $\Sq Z$ is interdefinable with the structure induced on $\rz$ by $\Sh N$ and $\Sq Z$ is o-minimal.
\end{prop}

\noindent
We expect that if $\Sa Z$ is dp-minimal then the structure induced on $\Z$ by $\Sq Z$ and $\psi_\alpha$ is interdefinable with $\Sh Z$.
All claims of Proposition~\ref{prop:complete-1} except o-minimality follow by slight modifications to the proof of Fact~\ref{fact:complete-0}.
The last claim also follows easily from the methods of \cite{big-nip}, we provide details below.
($\Sa H$ need not be o-minimal when $\Sa H$ is strongly dependent, for example $(\Q,+,<,\Z)$ is strongly dependent by \cite[3.1]{DoGo} and $(\Q,+,<,\Z)^\square$ is interdefinable with $(\R,+,<,\Z)$.)
\newline

\noindent
We need three facts to prove the last claim.
Fact~\ref{fact:boundary} is left to the reader.

\begin{fact}
\label{fact:boundary}
Suppose $X$ is a subset of $\R/\Z$.
Then $X$ is a finite union of intervals and singletons if and only if the boundary of $X$ is finite.
\end{fact}

\noindent
Fact~\ref{fact:accum} is essentially a theorem of Dolich and Goodrick~\cite{DoGo}.
They only treat linearly ordered structures, but routine alternations to their proof yield Fact~\ref{fact:accum}.

\begin{fact}
\label{fact:accum}
Suppose $(G,+,S)$ is a cyclically ordered abelian group, $\Sa G$ is a strongly dependent expansion of $(G,+,S)$, and $X$ is a $\Sa G$-definable subset of $G$.
If $X$ is nowhere dense then $X$ has no accumulation points.
\end{fact}

\noindent
Fact~\ref{fact:noiseless} follows from \cite[Theorem B]{big-nip}.

\begin{fact}
\label{fact:noiseless}
Suppose $\Sa Z$ is $\nip$ and $X,Y$ are $\Sq Z$-definable subsets of $(\rz)^n$.
Then $X$ either has interior in $Y$ or $X$ is nowhere dense in $Y$.
\end{fact}

\noindent
We now show that if $\Sa Z$ is strongly dependent then $\Sq Z$ is o-minimal.
We let $\bd(X)$ be the boundary of a subset $X$ of $\rz$.

\begin{proof}
Let $\Sa Z$ be strongly dependent and $X$ be an $\Sq Z$-definable subset of $\rz$.
By Fact~\ref{fact:noiseless} $X$ is not dense and co-dense in any interval.
So $\bd(X)$ is nowhere dense.
By Fact~\ref{fact:accum} $\bd(X)$ has no accumulation points, so $\bd(X)$ is finite by compactness of $\rz$.
By Fact~\ref{fact:boundary} $X$ is a finite union of intervals and singletons.
\end{proof}

\noindent
There is another way to show that $\Sq Z$ is o-minimal when $\Sa Z$ is dp-minimal.
Suppose $\Sa Z$ is dp-minimal.
Then $\Sh N$ is dp-minimal, so $\Sq Z$ is dp-minimal by Proposition~\ref{prop:complete-1}.
It follows from work of Simon~\cite{Simon-dp} that an expansion of $\rzc$ is o-minimal if and only if it is dp-minimal.
\newline

\noindent
In Section~\ref{section:zca-completion} we show $\zca^{\square}$ is interdefinable with $\rzc$ and $\shc$ is interdefinable with the structure induced on $\Z$ by $\zca^{\square}$ and $\psi_\alpha$.

\subsection{The completion of $(\Z,+,C_\alpha)$}
\label{section:zca-completion}
Proposition~\ref{prop:external} shows in particular that $(\Q,+,<)^\square$ is the usual completion of $(\Q,+,<)$.

\begin{prop}
\label{prop:external}
Suppose $H$ is a dense subgroup of $(\R,+)$.
Then $(H,+,<)^\square$ is interdefinable with $(\R,+,<)$.
\end{prop}

\noindent
Proposition~\ref{prop:external} will require the quantifier elimination for archimedean ordered abelian groups.
See Weispfennig~\cite{weis-oag} for a proof.

\begin{fact}
\label{fact:aoag-qe}
Let $(H,+,<)$ be an archimedean ordered abelian group.
Then $(H,+,<)$ admits quantifier elimination after adding a unary relation for every $nH$.
\end{fact}

\noindent
We now prove Proposition~\ref{prop:external}.
If $T : H^n \to H$ is a $\Z$-linear function given by $T(a_1,\ldots,a_n) = k_1 a_1 + \ldots + k_n a_n$ for integers $k_1,\ldots,k_n$ then we also let $T$ denote the function $\R^n \to \R$ given by $(t_1,\ldots,t_n) \mapsto k_1 t_1 + \ldots + k_n t_n$.

\begin{proof}
Let $(H + <) \prec (N,+,<)$ be highly saturated and let $\mfin, \st : \mfin^n \to \R^n$ be as above.
As $(H,+,<)$ is $\nip$, it suffices by Fact~\ref{fact:complete-0} to suppose that $Y \subseteq N^n$ is $\Sa N$-definable and show that $\st(Y \cap \mfin^n)$ is $(\R,+,<)$-definable.
If $Z$ is the closure of $Y$ in $N^n$ then $\st(Z \cap \mfin^n) = \st(Y \cap \mfin^n)$.
So we suppose that $Y$ is closed.
\newline

\noindent
As $Y$ is closed a straightforward application of Fact~\ref{fact:aoag-qe} shows that $Y$ is a finite union of sets of the form 
$$\{ a \in N^n : T_1(a) \leq s_1,\ldots,T_k(a) \leq s_k \}$$
for $\Z$-linear $T_1,\ldots,T_k : N^n \to N$ and $s_1,\ldots,s_k \in N$.
So we may suppose that $Y$ is of this form.
If $s_i > \mfin$ then $\mfin^n$ is contained in $\{ a \in N^n : T_i(a) \leq s_i \}$ and if $s_i < \mfin$ then $\{ a \in N^n : T_i(a) \leq s_i \}$ is disjoint from $\mfin^n$.
So we suppose $s_1,\ldots,s_k\in \mfin$.
It is now easy to see that
$$ \st(Y \cap \mfin^n) = \{ a \in \R^n : T_1(a) \leq \st(s_1), \ldots , T_k(a) \leq \st(s_k) \}.$$
So $\st(Y \cap \mfin^n)$ is $(\R,+,<)$-definable.
\end{proof}

\noindent
Proposition~\ref{prop:external-1} will be used to show that $\shc$ is interdefinable with the structure induced on $\Z$ by $\rzc$ and $\psi_\alpha$.

\begin{prop}
\label{prop:external-1}
Suppose $H$ is a dense subgroup of $(\R,+)$.
Then $(H,+,<)^{\mathrm{Sh}}$ is interdefinable with the structure induced on $H$ by $(\R,+,<)$.
\end{prop}

\begin{proof}

\noindent
As $(H,+,<)$ is $\nip$ Fact~\ref{fact:complete-0} shows that the structure induced on $H$ by $(\R,+,<)$ is a reduct of $(H,+,<)^{\mathrm{Sh}}$.
We show that $(H,+,<)^{\mathrm{Sh}}$ is a reduct of the structure induced on $H$ by $(\R,+,<)$.
Suppose $(H,+,<) \prec (N,+,<)$ is highly saturated and $Y \subseteq N^n$ is $(N,+,<)$-definable.
Applying Fact~\ref{fact:aoag-qe} there is a family $(X_{ij} : 1 \leq i,j \leq k)$ of $(N,+,<)$-definable sets such that $Y = \bigcup_{i = 1}^{k} \bigcap_{j = 1}^{k} X_{ij}$ and each $X_{ij}$ is either $(N,+)$-definable or of the form $\{ a \in N^n : T(a) \leq s \}$ for some $\Z$-linear $T : N^n \to N$ and $s \in N$.
As $H^n \cap Y = \bigcup_{i = 1}^{k} \bigcap_{j = 1}^{k} (H^n \cap X_{ij})$ it is enough to show that each $H^n \cap X_{ij}$ is definable in the structure induced on $H$ by $(\R,+,<)$.
If $X_{ij}$ is $(N,+)$-definable then $H^n \cap X_{ij}$ is $(H,+)$-definable by stability of abelian groups,
So suppose $Y = \{ a \in N^n : T(a) \leq s \}$ for $\Z$-linear $T : N^n \to N$ and $s \in N$.
Let $\mfin$ and $\st$ be as above.
If $s > \mfin$ then $H^n \subseteq Y$ and if $s < \mfin$ then $H^n$ is disjoint from $Y$.
Suppose $s \in \mfin$.
If $s \geq \st(s)$ then $H^n \cap Y = \{ a \in H^n : T(a) \leq \st(s) \}$ and if $s < \st(s)$ then $H^n \cap Y = \{ a \in H^n : T(a) < \st(x) \}$.
So in each case $H^n \cap Y$ is definable in the structure induced on $H$ by $(\R,+,<)$.
\end{proof}

\noindent
We can now compute $(\Z,+,C_\alpha)^\square$.

\begin{prop}
\label{prop:ca-complete}
Fix irrational $\alpha \in \rz$.
Then $(\Z,+,C_\alpha)^\square$ is interdefinable with $(\rz,+,C)$ and $\shc$ is interdefinable with the structure induced on $\Z$ by $(\rz,+,C)$ and $\psi_\alpha$.
\end{prop}

\begin{proof}
Let $\pi$ be the quotient map $\R \to \rz$ so $(\R,+,<,1,\pi)$ is a universal cover of $(\rz,+,C)$.
Fix $\lambda \in \R$ such that $\alpha = \lambda  + \Z$, let $H := \Z + \lambda \Z$, and let $\rho : H \to \Z$ be $\rho := \psi^{-1}_\alpha \circ \pi$, so that $(H,+,<,1,\rho)$ is a universal cover of $(\Z,+,C_\alpha)$.
Let $\rho : H^n \to \Z^n$ be given by $\rho(t_1,\ldots,t_n) = (\rho(t_1),\ldots,\rho(t_n))$.
Suppose $X \subseteq \Z^n$ is $(\Z,+,C_\alpha)^\mathrm{Sh}$-definable.
Then $Y := \rho^{-1}(X) \cap [0,1)^n$ is easily seen to be externally definable in $(H,+,<)$.
Proposition~\ref{prop:external} shows that $\cl(Y)$ is $(\R,+,<)$-definable.
Observe that $\pi(\cl(Y))$ is the closure of $\psi_\alpha(X)$ in $(\rz)^n$.
So the closure of $\psi_\alpha(X)$ in $(\rz)^n$ is definable in $(\rz,+,C)$.
So $(\Z,+,C_\alpha)^\square$ is interdefinable with $(\rz,+,C)$.
\newline

\noindent
We now show that $\shc$ is interdefinable with the structure induced on $\Z$ by $\rzc$ and $\psi_\alpha$.
By Proposition~\ref{prop:complete-1} and preceding paragraph it suffices to show that $\shc$ is a reduct of the induced structure.
Again suppose that $X \subseteq \Z^n$ is an $\shc$-definable subset of $\Z^n$ and $Y := \rho^{-1}(X) \cap [0,1)^n$.
By Proposition~\ref{prop:external-1} $Y$ is definable in the structure induced on $H \cap [0,1)$ by $(\R,+,<)$.
So $\rho(Y) = X$ is definable in the structure induced on $\psi_\alpha(\Z)$ by $\rzc$.
Hence $Y$ is definable in the structure induced on $\Z$ by $\rzc$ and $\psi_\alpha$.
\end{proof}

\noindent
Corollary~\ref{cor:shelah-induced} now follows immediately, we leave the details to the reader.

\begin{cor}
\label{cor:shelah-induced}
Suppose $\bH$ is a semialgebraic Mordell-Lang circle group, $\gamma$ is the unique (up to sign) topological group isomorphism $\rz \to \bH$, $\alpha \in \rz$ is irrational, $\chi : \Z \to \bH$ is the character $\chi(k) := \gamma(\alpha k)$, and $\Sa H_\alpha$ is the structure induced on $\Z$ by $\rfield)$ and $\chi$.
Then $\Sa H_\alpha$ expands $\shc$.
So in particular $\Sa G_{\alpha,\lambda}, \Sa S_\alpha$, and $\Sa E_{\alpha,\eta}$ all expand $\shc$ for any $\lambda,\eta > 1$.
\end{cor}

\noindent
By Fact~\ref{fact:simon-w} a dp-minimal expansion of $\shc$ cannot add new unary sets.
We suspect that any dp-minimal expansion of $\shc$ adds new binary sets.

\begin{prop}
\label{prop:new-binary}
Fix irrational $\alpha \in \rz$.
Then $\Sa G_{\alpha,\lambda}, \Sa S_\alpha$, and $\Sa E_{\alpha,\eta}$ all define a subset of $\Z^2$ which is not $\shc$-definable for any $\lambda,\eta > 1$.
\end{prop}

\noindent
An open subset of a topological space is \textbf{regular} if it is the interior of its closure.

\begin{proof}
We treat $\Sa G_{\alpha,\lambda}$, the other cases follow in the same way.
Let $S$ the cyclic order on $[1,\lambda)$ where $S(t,t',t'')$ if and only if either $t < t' < t''$, $t' < t'' < t$ or $t'' < t < t'$.
So $S$ is the unique (up to opposite) semialgebraic cyclic group order on $([1,\lambda),\otimes_\lambda)$.
Let $U$ be a regular open semialgebraic subset of $[1,\lambda)^2$ which is not definable in $([1,\lambda),\otimes_\lambda,S)$, e.g. an open disc contained in $[1,\lambda)^2$.
Let $V : = \chi^{-1}_\alpha(U)$, so $V$ is $\Sa G_{\alpha,\lambda}$-definable.
Suppose that $V$ is $\shc$-definable.
By Proposition~\ref{prop:ca-complete} the closure of $\chi_\alpha(V)$ is definable in $([1,\lambda),\otimes_\lambda,S)$.
As $\chi_\alpha(\Z^2)$ is dense in $[1,\lambda)^2$, the closure of $\chi_\alpha(V)$ agrees with the closure of $U$.
As $U$ is regular $U$ is the interior of the closure of $U$.
So $U$ is definable in $([1,\lambda),\otimes_\lambda,S)$, contradiction.
\end{proof}

\section{When the examples are interdefinable}
\label{section:interdef}
\noindent
In this section we describe the completions of the  dp-minimal expansions of $\zca$ constructed in Section~\ref{section:examples} and show that if two of these expansions are interdefinable then the associated semialgebraic circle groups are semialgebracially isomorphic.
\newline

\noindent
Suppose $\Sa R$ is an o-minimal expansion of $\rfield)$, $\bH$ is an $\Sa R$-definable circle group.
We say that a subgroup $A$ of $\bH$ is a \textbf{GH-subgroup} if $(\Sa R,A)$ is $\nip$, $\Th(\Sa R)$ is an open core of $\Th(\Sa R,A)$, and the structure induced on $A$ by $\Sa R$ is dp-minimal.

\begin{prop}
\label{prop:induce-complete}
Suppose $\Sa R$ is an o-minimal expansion of $\rfield)$,  $\bH$ is an $\Sa R$-definable circle group,
$\gamma$ is the unique (up to sign) topological group isomorphism $\rz \to \bH$, $\chi$ is an injective character $\Z \to \bH$, and $\Sa Z$ is the structure induced on $\Z$ by $\Sa R$ and $\chi$.
If $\chi(\Z)$ is a $\gh$-subgroup then $\Sq Z$ is interdefinable with the structure induced on $\rz$ by $\Sa R$ and $\gamma$.
So for any irrational $\alpha \in \rz$ and $\lambda > 1$:
\begin{enumerate}
\item A subset of $(\rz)^n$ is $\Sq G_{\alpha,\lambda}$-definable if and only if it is the image under the quotient map $\R^n \to (\rz)^n$ of a set of the form $$\{ ( t_1,\ldots, t_n) : (\lambda^{t_1},\ldots,\lambda^{t_n}) \in X \}$$ for a semialgebraic subset $X$ of $[1,\lambda)^n$.
\item A subset of $(\rz)^n$ is $\Sq S_\alpha$-definable if and only if it is the image under the quotient map $\R^n \to (\rz)^n$ of a set of the form
$$ \{ (t_1,\ldots,t_n) \in [0,1)^n : (e^{2\pi i t_1},\ldots,e^{2\pi i t_n}) \in X \}  $$
for a semialgebraic subset $X$ of $\bS^n$.
\item A subset of $(\rz)^n$ is $\Sq E_{\alpha,\lambda}$-definable if and only if it is an image under the quotient map $\R^n \to (\rz)^n$ of a set of the form
$$ \{ (t_1,\ldots,t_n) \in [0,1)^n : (\frp(t_1),\ldots,\frp(t_n) ) \in X \} $$
for a semialgebraic subset $X$ of $\E^0_\lambda(\R)^n$.
\end{enumerate}
\end{prop}

\noindent
It follows from Proposition~\ref{prop:induce-complete} that if $\Sa Z$ is one of the expansions of $\zca$ described above then $\Sq Z$ defines an isomorphic copy of $\rfield)$, so if $\Sa Z \prec \Sa N$ is highly saturated then $\Sh N$ interprets $\rfield)$.
So $\Sa Z$ is non-modular.
An adaptation of \cite[Proposition 15.2]{big-nip} shows that $\Sa N$ cannot interpret an infinite field.
\newline

\noindent
We prove Theorem~\ref{thm:complete-induced}, a more general result on completions which covers almost all ``dense pairs".
It is easy to see that Proposition~\ref{prop:induce-complete} follows from Theorem~\ref{thm:complete-induced}, we leave the details of this to the reader.

\begin{thm}
\label{thm:complete-induced}
Let $\Sa S$ be an o-minimal expansion of $(\R,+,<)$.
Suppose $A$ is a subset of $\R^m$ such that $(\Sa S,A)$ is $\nip$ and $\mathrm{Th}(\Sa S)$ is an open core of $\mathrm{Th}(\Sa S,A)$.
Let $\Sa A$ be the structure induced on $A$ by $\Sa S$ and $X$ be the closure of $A$ in $\R^m$.
Then
\begin{enumerate}
\item the structure $\Sq A$ with domain $X$ and an $n$-ary relation symbol defining $\cl(Y)$ for each $\Sh A$-definable $Y \subseteq A^n$.
\item and the structure $\Sa X$ induced on $X$ by $\Sa S$,
\end{enumerate}
are interdefinable.
(Note that $X$ is $\Sa S$-definable.)
\end{thm}

\noindent
We let
$$ \| a \| := \max \{ |a_1|,\ldots,|a_n| \} \quad \text{for all  } a = (a_1,\ldots,a_n) \in \R^n. $$
We will need a metric argument from \cite{big-nip} to show that $\Sa X$ is a reduct of $\Sq A$.
If $X = \R^m$ then one can can give a topological proof following \cite[Proposition 3.4]{field-in-Shelah}.

\begin{proof}
\noindent
We first show that $\Sa X$ is a reduct of $\Sq A$.
Suppose $Y$ is a nonempty $\Sa S$-definable subset of $X^n$.
By o-minimal cell decomposition there are definable closed subsets $E_1,F_1\ldots,E_k,F_k$ of $\R^{nm}$ such that $Y = \bigcup_{i = 1}^{k} (E_i \setminus F_i)$.
We have
$$Y = \bigcup_{i = 1}^{k} ((X^n \cap  E_i) \setminus (X^n \cap F_i))$$
so we may suppose that $Y$ is a nonempty closed $\Sa S$-definable subset of $X^n$.
Let $W$ be the set of $(a,a',c) \in X \times X \times X^n$ for which there is $c' \in Y$ satisfying $\| c - c' \| < \| a - a' \|$.
So $W \cap (A \times A \times A^n)$ is $\Sa A$-definable and $Z := \cl(W \cap (A \times A \times A^n))$ is $\Sq A$-definable.
The metric argument in the proof of \cite[Lemma 13.5]{big-nip} shows that
$$ Y = \bigcap_{a,a' \in X, a \neq a'} \{ c \in X : (a,a',c) \in Z \}.$$
(This metric argument requires $Y$ to be closed.)
So $Y$ is $\Sq A$-definable.
\newline

\noindent
We now show that $\Sa X$ is a reduct of $\Sq A$.
Suppose $Y$ is an $\Sh A$-definable subset of $A^n$.
We show that $\cl(Y) \subseteq X^n$ is $\Sa S$-definable.
As $(\Sa S,A)$ is $\nip$, $\Sa A$ is $\nip$, so an application of Fact~\ref{fact:cs-limit} yields an $\Sa A$-definable family $(Y_a : a \in A^k)$ of subsets of $A^n$ such that for every finite $B \subseteq Y$ we have $B \subseteq Y_a \subseteq Y$ for some $a \in A^k$.
As $\mathrm{Th}(\Sa S)$ is an open core of $\mathrm{Th}(\Sa S,A)$ there is an $\Sa S$-definable family $(Z_b : b \in \R^l)$ of subsets of $\R^{nm}$ such that for every $a \in A^k$ we have $\cl(Y_a) = Z_b$ for some $b \in \R^l$.
So for every finite $F \subseteq X$ there is $b \in \R^l$ such that $F \subseteq Z_b \subseteq \cl(Y)$.
A saturation argument yields an $\Sh R$-definable subset $Z$ of $X^n$ such that $Y \subseteq Z \subseteq \cl(Y)$.
An application of Fact~\ref{fact:ms} shows that $Z$ is $\Sa S$-definable, so $\cl(Z) = \cl(Y)$ is $\Sa S$-definable.
\end{proof}

\noindent
The proof of Theorem~\ref{thm:complete-induced} goes through for any expansion $\Sa S$ of $(\R,+,<)$ such that $\Sa S$ is $\nip$, every $\Sa S$-definable set is a boolean combination of definable closed sets, and $\Sh S$ is interdefinable with $\Sa S$.
So for example Theorem~\ref{thm:complete-induced} holds when $\Sa S = (\R,+,<,\Z)$.
\newline

\noindent
Our next goal is to show that if $\bH$ and $\Sa Z$ are as in Proposition~\ref{prop:induce-complete} then we can recover $\Sa R$ and $\bH$ from $\Sa Z$.
We show that we can recover $\Sa R$ and $\bH$ from $\Sq Z$.
This follows from a general correspondence between
\begin{enumerate}
\item non-modular o-minimal expansions $\Sa C$ of $\rzc$, and
\item pairs of the form $\la \Sa R, \bH \ra$, for an o-minimal expansion $\Sa R$ of $(\R,+,\times)$ and an $\Sa R$-definable circle group $\bH$.
\end{enumerate}
In this correspondence $\Sa C$ is unique up to interdefinibility, $\Sa R$ is unique up to interdefinibility, and $\bH$ is unique up to $\Sa R$-definable isomorphism.
\newline

\noindent
Suppose that $\Sa R$ is an o-minimal expansion of $(\R,+,\times)$ and $\bH$ is an $\Sa R$-definable circle group.
We consider $\bH$ as a topological group with $\Cal T_\bH$.
Let $\gamma$ be the unique (up to sign) topological group isomorphism $\rz \to \bH$.
Let $\Sa C$ be the structure induced on $\rz$ by $\Sa R$ and $\gamma$.
Note $\Sa C$ is unique up to interdefinibility.
It is easy to see that $\Sa C$ defines an isomorphic copy of $\rfield)$.
\newline

\noindent
Now suppose $\Sa C$ is a non-modular o-minimal expansion of $\rzc$.
Suppose $I$ is a non-empty open interval and $\oplus,\otimes : I^2 \to I$ are $\Sa C$-definable such that $(I,\oplus,\otimes)$ is isomorphic to $\rfield)$.
Let $\iota$ be the unique isomorphism $\rfield) \to (I,\oplus,\otimes)$.
Let $\Sa R$ be the structure induced on $\R$ by $\Sa C$ and $\iota$.
By compactness of $\rz$ there is a finite $A \subseteq \rz$ such that $(a + I : a \in A)$ covers $\rz$.
Fix a bijection $f : B \to A$ for some $B \subseteq \R$.
Let $\tau : B \times \R \to \rz$ be the surjection given by $\tau(b,t) = f(b) + \iota(t)$.
Observe that equality modulo $\tau$ is an $\Sa R$-definable equivalence relation and, applying definable choice, let $H$ be an $\Sa R$-definable subset of $B \times \R$ which contains one element from each fiber of $\tau$.
Let $\tau' : H \to \rz$ be the induced bijection and  $\boxplus$ be the pullback of $+$ by $\tau'$.
Then $\bH := (H,\boxplus)$ is an $\Sa R$-definable circle group.
Note that the expansion of $\rzc$ associated to $\la \Sa R, \bH \ra$ is interdefinable with $\Sa C$.


\begin{prop}
\label{prop:circle-unique}
For $i \in \{0,1\}$ suppose that $\Sa R_i$ is an o-minimal expansion of $\rfield)$, $\bH_i$ is an $\Sa R_i$-definable circle group, and $\Sa C_i$ is the expansion of $\rzc$ associated to $\la \Sa R_i, \bH_i \ra$.
If $\Sa C_0$ and $\Sa C_1$ are interdefinable then $\Sa R_0$ and $\Sa R_1$ are interdefinable and there is an $\Sa R_0$-definable group isomorphism $\bH_0 \to \bH_1$.
\end{prop}

\begin{proof}
\noindent
It is easy to see that $\Sa C_0$ is bi-interpretable with $\Sa R_0$ and $\Sa C_1$ is bi-interpretable with $\Sa R_1$.
So if $\Sa C_0$ and $\Sa C_1$ are interdefinable then $\Sa R_0$ and $\Sa R_1$ are bi-interpretable, hence interdefinable by Fact~\ref{fact:opp}.
So we suppose $\Sa R_0 = \Sa R_1$ and denote $\Sa R_0$ by $\Sa R$.
\newline

\noindent
For each $i \in \{0,1\}$ let $I_i$ be an interval in $\rz$ and $\iota_i$ be a bijection $\R \to I_i$ such that $\Sa R$ is interdefinable with the structure induced on $\R$ by $\Sa C_i$ and $\iota_i$.
Let $\Sa F_i$ be the pushforward of $\Sa R$ by $\iota_i$ for $i \in \{0,1\}$.
So $\Sa F_0$ is a $\Sa C_0$-definable copy of $\Sa R$ and $\Sa F_1$ is a $\Sa C_1$-definable copy of $\Sa R$.
Let $\bH_{00}, \bH_{01}$ be the pushforward of $\bH_0, \bH_1$ by $\iota_0$, respectively.
Likewise, let $\bH_{10}, \bH_{11}$ be the pushforward of $\bH_0, \bH_1$ by $\iota_1$, respectively.
So $\bH_{00}$, $\bH_{01}$ are $\Sa F_0$-definable copies of $\bH_0$, $\bH_1$, respectively, and $\bH_{10}$, $\bH_{11}$ are $\Sa F_1$-definable copies of $\bH_0$, $\bH_1$, respectively.
Given $i \in \{0,1\}$ let $\gamma_i$ be a $\Sa C_i$-definable group isomorphism $\bH_{i0} \to \rz$.
(Note that $\Sa C_0$ a priori does not
define a group isomorphism from $\bH_{01}$ to $\rz$, likewise for $\Sa C_1$ and $\bH_{10}$.)
\newline

\noindent
Now suppose that $\Sa C_0$ and $\Sa C_1$ are interdefinable.
We show that $\bH_0$ and $\bH_1$ are $\Sa R$-definably isomorphic.
It suffices to show that $\bH_{00}$ and $\bH_{01}$ are $\Sa F_0$-definably isomorphic.
As $\Sa F_0$ and $\Sa C_0$ are bi-interpretable it is enough to produce a $\Sa C_0$-definable group isomorphism $\bH_{00} \to \bH_{01}$.
As $\Sa C_0$ and $\Sa C_1$ are interdefinable $\gamma^{-1}_1 \circ \gamma_0$ is a $\Sa C_0$-definable group isomorphism $\bH_{00} \to \bH_{11}$.
By Fact~\ref{fact:opp} there is a $\Sa C_0$-definable bijection $\xi : I_1 \to I_0$ which induces an isomorphism (up to interdefinibility) from $\Sa F_1$ to $\Sa F_0$.
Let $\zeta$ be the $\Sa C_0$-definable group isomorphism $\bH_{11} \to \bH_{01}$ induced by $\xi$.
Then $\zeta \circ \gamma_1 \circ \gamma^{-1}_0$ is a $\Sa C_0$-definable group isomorphism $\bH_{00} \to \bH_{01}$.
\end{proof}

\noindent
Theorem~\ref{thm:combine-1} classifies our examples up to interdefinibility.

\begin{thm}
\label{thm:combine-1}
Let $\Sa R_0,\Sa R_1,\bH_0,\bH_1$ be as in Proposition~\ref{prop:circle-unique}.
Fix irrational $\alpha \in \rz$.
For each $i \in \{0,1\}$ let  $\gamma_i : \rz \to \bH_i$ be the unique (up to sign) topological group isomorphism, $\chi_i : \Z \to \bH_i$ be given by $\chi_i(k) := \gamma_i(\alpha k)$, and $\Sa Z_i$ be the structure induced on $\Z$ by $\Sa R_i$ and $\chi_i$.
Suppose $\chi_i(\Z)$ is a $\gh$-subgroup for $i \in \{0,1\}$.
Then $\Sa Z_0$ and $\Sa Z_1$ are interdefinable if and only if $\Sa R_0$ and $\Sa R_1$ are interdefinable and there is an $\Sa R_0$-definable group isomorphism $\bH_0 \to \bH_1$.
\end{thm}

\begin{proof}
Suppose that $\Sa R_0$ and $\Sa R_1$ are interdefnable and $\xi : \bH_0 \to \bH_1$ is an $\Sa R_0$-definable group isomorphism.
Then $\xi \circ \gamma_0$ is the unique (up to sign) topological group isomorphism $\rz \to \bH_1$.
So after possibly replacing $\xi$ with $-\xi$ we have $\gamma_1 = \xi \circ \gamma_0$, hence $\chi_1 = \xi \circ \chi_0$.
It easily follows that $\Sa Z_0$ and $\Sa Z_1$ are interdefinable.
\newline

\noindent
Suppose $\Sa Z_0$ and $\Sa Z_1$ are interdefinable.
Then $\Sq Z_0$ and $\Sq Z_1$ are interdefinable.
By Proposition~\ref{prop:induce-complete} the expansions of $\rzc$ associated to $\la \Sa R_0, \bH_0 \ra$ and $\la \Sa R_1, \bH_1 \ra$ are interdefinable.
Applying Proposition~\ref{prop:circle-unique} see that $\Sa R_0$ and $\Sa R_1$ are interdefinable and there is an $\Sa R_0$-definable group isomorphism $\bH_0 \to \bH_1$.
\end{proof}

\noindent
We now see that we have constructed uncountably many dp-minimal expansions of each $\shc$.
Corollary~\ref{cor:to-Z} follows from Theorem~\ref{thm:combine-1} and the classification of one-dimensional semialgebraic groups described above.

\begin{cor}
\label{cor:to-Z}
Fix irrational $\alpha \in \rz$ and let $\lambda,\eta > 1$.
Then
\begin{enumerate}
\item no two of $\Sa G_{\alpha,\lam}, \Sa S_\alpha, \Sa E_{\alpha,\eta}$ are interdefinable,
\item $\Sa G_{\alpha,\lambda}$ and $\Sa G_{\alpha,\eta}$ are interdefinable if and only if $\lambda/\eta \in \Q$,
\item $\Sa E_{\alpha,\lambda}$ and $\Sa E_{\alpha,\eta}$ are interdefinable if and only if $\lambda/\eta \in \Q$.
\end{enumerate}
\end{cor}

\noindent
Suppose for the rest of this section that Conjecture~\ref{conj:le-gal} holds.
Fix irrational $\alpha \in \rz$.
Suppose that $\bH$ is a semialgebraic Mordell-Lang circle group, $\gamma : \rz \to \bH$ is the unique (up to sign) topological group isomorphism, and $\chi : \Z \to \bH$ is the character $\chi(k) := \gamma(\alpha k)$.
Let $\Sa H_\alpha$ be the structure induced on $\Z$ by $\rfield)$ and $\chi$.
For any $f \in \Lambda$ let $\Sa H_{\alpha,f}$ be the structure induced on $\Z$ by $\rfield,f)$ and $\chi$.
Then $\Sa H_{\alpha,f}$ is dp-minimal and $\Sq H_{\alpha,f}$ is interdefinable with the structure induced on $\rz$ by $\rfield,f)$ and $\gamma$.
It follows that by Proposition~\ref{prop:induce-complete} that $\Sa H_{\alpha,f}$ is a proper expansion of $\Sa H_{\alpha}$ and if $f,g$ are distinct elements of $\Lambda$ then $\Sa H_{\alpha,f}$ and $\Sa H_{\alpha,g}$ are not interdefinable.
\newline

\noindent
In this way, still assuming Conjecture~\ref{conj:le-gal}, we can produce produce two dp-minimal expansions of $\Sa H_{\alpha}$ which do not have a common $\nip$ expansion.
Let $h \in C^\infty(I)$ be such that $\rfield,h)$ is not $\nip$.
(For example one can arrange that $I = [0,1]$ and $\{ t \in I : f(t) = 0 \}$ is $0 \cup \{ 1/n : n \geq 1 \}$.)
As $\Lambda$ is comeager an application of the Pettis lemma~\cite[Theorem 9.9]{kechris} implies that there are $f,g \in \Lambda$ and $t > 0$ such that $f - g = th$.
So after rescaling $h$ we suppose $f - g = h$.
Suppose that $\Sa Z$ is an $\nip$ expansion of both $\Sa H_{\alpha,f}$ and $\Sa H_{\alpha,g}$.
Then $\Sq Z$ is $\nip$.
An easy argument using the first part of the proof of Theorem~\ref{thm:complete-induced} shows that $\rfield,f,g)$ is interpretable in $\Sq Z$, contradiction.
(This kind of argument was previously used by Le Gal~\cite{LGal} to show that there are two o-minimal expansions of $\rfield)$ which are not reducts of a common o-minimal structure.)

\section{Dp-minimal expansions of $\zpp$}
\label{section:p-adic}
\noindent
Throughout $p$ is a fixed prime.
To avoid mild technical issues we assume $p \neq 2$.
(Add a reference for dp-minimality of $\Z_p$.)

\subsection{A proper dp-minimal expansion of $\zpp$}
We apply work of Mariaule.
The first and third claims of Fact~\ref{fact:mari-main} are special cases of the results of \cite{Ma-adic}.
The second claim follows from Mariaule's results and a general theorem of Boxall and Hieronymi on open cores~\cite{BoxallH}.
Recall that $1 + p\Z_p$ is a subgroup of $\Z^\times_p$.

\begin{fact}
\label{fact:mari-main}
Suppose that $A$ is a finitely generated dense subgroup of $(1 + p\Z_p,\times)$.
Then $\pring,A)$ is $\nip$, $\Th\pring)$ is an open core of $\Th\pring,A)$, and every $\pring,A)$-definable subset of $A^k$ is of the form $X \cap Y$ where $X$ is an $(A,\times)$-definable subset of $A^k$ and $Y$ is a semialgebraic subset of $\Z^k_p$.
\end{fact}

\noindent
We let $\pexp$ be the $p$-adic exponential, i.e.
$$ \pexp(a) := \sum_{n = 0}^{\infty} \frac{a^n}{n!} \quad \text{for all  } a \in p\Z_p.$$
(The sum does not converge off $p\Z_p$.)
$\pexp$ is a topological group isomorphism $(p\Z_p,+) \to (1 + p\Z_p,\times)$. 
So $a \mapsto \pexp(pa)$ is a topological group isomorphism $(\Z_p,+) \to (1 + p\Z_p,\times)$.
It is easy to see that
$$\valp(\pexp(a) - 1) = \valp(a) \quad \text{for all  } a \in p\Z_p.$$
So for all $a \in \Z_p$ we have
$$ \valp(\pexp(pa) - 1) = \valp(pa) = \valp(a) + 1. $$
Define $\mathbf{v}(b) = \valp(b - 1) - 1$ for all $b \in 1 + p\Z_p$, so $a \mapsto \pexp(pa)$ is an isomorphism $(\Z_p,+,\valp) \to (1 + p\Z_p, \times, \mathbf{v})$.
\newline

\noindent
We let $\chi : \Z \to \Z^{\times}_p$ be the character $\chi(k) := \pexp(pk)$ and let $\Sa P$ be the structure induced on $\Z$ by $\pring)$ and $\chi$.
Note that $\Sa P$ expands $\zpp$ because $\chi$ is an isomorphism $\zpp \to (\chi(\Z),\times,\mathbf{v})$.
Let
$$ \chi(k_1,\ldots,k_n) = (\chi(k_1),\ldots,\chi(k_n)) \quad \text{for all  } (k_1,\ldots,k_n) \in \Z^n. $$

\noindent
There are $\Sa P$-definable subsets of $\Z$ which are not $\zpp$-definable.
Consider $\N \cup \{\infty\}$ as the value set of $\zpp$.
It follows from the quantifier elimination for $\zpp$ that the structure induced on $\N \cup \{\infty\}$ by $\zpp$ is interdefinable with $(\N \cup \{\infty\}, <)$.
So $\valp^{-1}(2\N)$ is $\Sa P$-definable and not $\zpp$-definable.
Let $E$ be the set of $a \in 1 + p\Z_p$ such that $\mathbf{v}(a) \in 2\N$.
Note that if $b,c \in 1 + p\Z_p$ then $\chi^{-1}(bE) = \chi^{-1}(cE)$ if and only if $b = c$.
So $\Sa P$ defines uncountably many subsets of $\Z$, in constrast $\zpp$ defines only countably many subsets of $\Z$.

\begin{prop}
\label{prop:p-adic-char}
$\Sa P$ is dp-minimal.
\end{prop}

\noindent
Proposition~\ref{prop:p-adic-char} requires some preliminaries.
A formula $\vartheta(x;y)$ is \textbf{bounded} if  $\Sa P \models \forall y \exists^{\leq n} x \vartheta(x;y)$ for some $n$.
Let $\lab$ be the language of abelian groups together with unary relations $(D_n)_{ n \geq 1}$ and $\lval$ be the expansion of $\lab$ by a binary relation $\preccurlyeq_p$
We let $D_n$ define $n\Z$ and declare $k \preccurlyeq_p k'$ if and only if $\valp(k) \leq \valp(k')$.
Fact~\ref{fact:qe-valp} was independently proven by Alouf and d'Elb\'{e}e~\cite{AldE}, Mariaule~\cite{Mar}, and Guignot~\cite{Guinot}.

\begin{fact}
\label{fact:qe-valp}
$\zpp$ has quantifier elimination in $\lval$.
\end{fact}

\noindent We let $\linduce$ be the language with an $n$-ary relation symbol defining $\chi^{-1}(X)$ for each semialgebraic $X \subseteq \Z^n_p$.
So $Y \subseteq \Z^n$ is quantifier free $\linduce$-definable if and only if $Y = \chi^{-1}(X)$ for semialgebraic $X \subseteq \Z^n_p$.
Take $\Sa P$ to be an $(\lval \cup \linduce)$-structure.
\newline

\noindent
We first give a description of unary $\Sa P$-definable sets.

\begin{lem}
\label{lem:split}
Suppose $\varphi(x;y), |x| = 1$ is an $(\lval \cup \linduce)$-formula. 
Then $\varphi(x;y)$ is equivalent to a finite disjunction of formulas of the form $\varphi_1(x;y) \land \varphi_2(x;y)$ where $\varphi_2(x;y)$ is a quantifier free $\linduce$-formula and $\varphi_1(x;y)$ is an $\lab$-formula such that either $\varphi_1(x;y)$ is bounded or there are integers $k, l$ such that $\valp(k) = 0$ and for every $b \in \Z^{|y|}$, $\varphi_1(\Z;b) = (k\Z + l) \setminus A$ for finite $A$.
\end{lem}

\noindent
The condition $\valp(k) = 0$ ensures that each $\varphi_1(\Z;b)$ is dense in the $p$-adic topology.

\begin{proof}
By Fact~\ref{fact:mari-main} we may suppose $\varphi(x;y) = \varphi_1(x;y) \land \varphi_2(x;y)$ where $\varphi_1$ is an $\lval$-formula and $\varphi_2$ is a quantifier free $\linduce$-formula.
By Fact~\ref{fact:qe-valp} we may suppose
$$ \varphi_1(x;y) = \bigvee_{i = 1}^{m} \bigwedge_{j = 1}^{m} \vartheta_{ij}(x;y)$$
where each $\vartheta_{ij}(x;y)$ is an atomic $\lval$-formula.
So we may suppose
$$ \varphi(x;y) = \bigvee_{i = 1}^{m} \left( \varphi_2(x;y) \land \bigwedge_{j = 1}^{m} \vartheta_{ij}(x;y) \right). $$
So we suppose $\varphi(x;y)$ is of the form $\varphi_2(x;y) \land \bigwedge_{j = 1}^{m} \vartheta_{j}(x;y)$ where each $\vartheta_j(x;y)$ is an atomic $\lval$-formula.
After possibly rearranging there is $0 \leq m' \leq m$ such that
\begin{enumerate}
\item if $1 \leq j \leq m'$ then $\vartheta_j(x;y)$ is of the form $g \preccurlyeq_p h$ or $g \prec_p h$ where $g,h$ are $\lab$-terms in the variables $x,y$, and
\item if $m' < j \leq m$ then $\vartheta_j(x;y)$ is an atomic $\lab$-formula.
\end{enumerate}
Note that any formula of type $(1)$ is equivalent to a quantifier free $\linduce$-formula.
Now $$ \left( \varphi_2(x;y) \land \bigwedge_{j = 1}^{m'} \vartheta_j(x;y) \right) \land \bigwedge_{j = m'+1}^{m} \vartheta_j(x;y).$$
The formula inside the parentheses is equivalent to a quantifier free $\linduce$-formula.
So we suppose that $\varphi(x;y)$ is for the form $\varphi_1(x;y) \land \varphi_2(x;y)$ where $\varphi_1(x;y)$ is an $\lab$-formula and $\varphi_2(x;y)$ is an quantifier free $\linduce$-formula.
An easy application of quantifier elimination shows that $\varphi_1(x;y)$ is equivalent to a formula of the form $\bigvee_{i  = 1}^{m} \theta_i(x;y)$ 
where for each $i$ either:
\begin{enumerate}
\item $\theta_i(x;y)$ is bounded, or
\item there are integers $k \neq 0,l$ and an $\lab$-formula $\theta'_i(x;y)$ such that
$\theta_i(x;y)$ is equivalent to $(x \in k'\Z + l) \land \neg\theta'_i(x;y)$ and $\theta'_i(x;y)$ is bounded.
\end{enumerate}
Applying the same reasoning as above we may suppose that $\varphi_1(x;y)$ satisfies $(1)$ or $(2)$ above.
If $\varphi_1(x;y)$ is bounded then we are done.
So fix integers $k',l$ and bounded $\varphi'_1(x;y)$ such that $\varphi_1(x;y)$ is equivalent to $(x \in k'\Z + l) \land \neg \varphi'_1(x;y)$.
Let $v := \valp(k')$ and $k := k'/p^v$.
So $k'\Z + l = (p^v\Z + l) \cap (k\Z + l)$ and $\varphi_1(x;y)$ is logically equivalent to 
$$ [\valp(x - l) \geq v] \land [x \in k\Z + l] \land \neg\varphi'_1(x;y). $$
After replacing $\varphi_1(x;y)$ with $[x \in k\Z + l] \land \neg \varphi'_1(x;y)$ and replacing $\varphi_2(x;y)$ with $[\valp(x - l) \geq v] \land \varphi_2(x;y)$ we may suppose that for every $b \in \Z^{|y|}$, $\varphi_1(\Z;b)$ agrees with $(k\Z + l) \setminus A$ for finite $A$.
\end{proof}

\noindent
We also need Fact~\ref{fact:vis}, a consequence of the quantifier elimination for $\pring)$.

\begin{fact}
\label{fact:vis}
Suppose that $\phi(x;y), |x| = 1$ is a formula in the language of rings.
Then there are formulas $\phi_1(x;y), \phi_2(x;y)$ such that 
\begin{enumerate}
\item $\phi(x;y)$ and $\phi_1(x;y) \vee \phi_2(x;y)$ are equivalent in $\pring)$,
\item $\phi_1(\Z_p;b)$ is finite and $\phi_2(\Z_p;b)$ is open for every $b \in \Z^{|y|}_p$.
\end{enumerate}
\end{fact}
\noindent

\noindent
Lemma~\ref{lem:vis-1} follows from inp-minimality of $\pring)$ and the fact that $\chi(X)$ is dense in $1 + p\Z_p$.
We leave the verification to the reader.

\begin{lem}
\label{lem:vis-1}
Suppose that $\varphi(x;y), \phi(x;y), |x| = 1$ are quantifier free $\linduce$-formulas such that $\varphi(\Z;b)$ and $\phi(\Z;b)$ are both open in the $p$-adic topology for every $b \in \Z^{|y|}$.
Then $\varphi(x;y)$ and $\phi(x;y)$ cannot violate inp-minimality.
\end{lem}

\noindent
We now prove Proposition~\ref{prop:p-adic-char}.

\begin{proof}
We equip $\Z$ with the $p$-adic topology.
Fact~\ref{fact:mari-main} shows that $\Sa P$ is $\nip$ so it is enough to show that $\Sa P$ is inp-minimal.
Suppose towards a contradiction that $\varphi(x;y)$, $\phi(x;z)$, and $n$ violate inp-minimality.
Applying Lemma~\ref{lem:split}, Fact~\ref{fact:union-reduce}, and Fact~\ref{fact:remove-finite} we may suppose there are $\varphi_1(x;y), \varphi_2(x;y)$ and $k,l$ such that
\begin{enumerate}
\item $\varphi(x;y) = \varphi_1(x;y) \land \varphi_2(x;y)$,
\item $\varphi_2(x;y)$ is a quantifier free $\linduce$-formula, and
\item $\valp(k) = 0$ and for all $b \in \Z^{|y|}$, $\varphi_1(\Z;b) = (k\Z + l) \setminus A$ for finite $A$.
\end{enumerate}
Applying Fact~\ref{fact:vis} we get $\linduce$-formulas $\varphi'_2(x;y)$ and $\varphi''_2(x;y)$ such that $\varphi'_2(x;y)$ is bounded, $\varphi''_2(\Z;b)$ is open for all $b \in \Z^{|y|}$, and $\varphi_2(x;y) = \varphi'_2(x;y) \vee \varphi''_2(x;y)$.
Applying Facts~\ref{fact:union-reduce} and \ref{fact:remove-finite} we may suppose that $\varphi_2(\Z;b)$ is open for all $b \in \Z^{|y|}$.
\newline

\noindent
We have reduced to the case when $\varphi(x;y) = \varphi_1(x;y) \land \varphi_2(x;y)$ where $\varphi_2(x;y)$ is a quantifier free $\linduce$-formula such that each $\varphi_2(\Z;b)$ is open and there are $k,l$ such that $\valp(k) = 0$ and for all $b \in \Z^{|y|}$ we have $\varphi_1(\Z;b) = (k\Z + l) \setminus A$ for finite $A$.
By the same reasoning we may suppose that there are $\phi_1(x;z), \phi_2(x;z)$ and $k',l'$ which satisfy the same conditions with respect to $\phi(x;z)$.
\newline

\noindent
We show that $\varphi_2(x;y), \phi_2(x;z)$ and $n$ violate inp-minimality and thereby obtain a contradiction with Lemma~\ref{lem:vis-1}.
Fix $a_1,\ldots,a_m \in \Z^{|y|}$ and $b_1,\ldots,b_m \in \Z^{|z|}$  such that $\varphi(x;a_1),\ldots,\varphi(x;a_m)$ and $\phi(x;b_1),\ldots,\phi(x;b_m)$ are both $n$-inconsistent and $\Sa P \models \exists x \varphi(x;a_i) \land \phi(x;a_j)$ for all $i,j$.
So $\Sa P \models \exists x \varphi_2(x;a_i) \land \phi_2(x;b_j)$ for all $i,j$.
It suffices to show that $\varphi_2(x;a_1),\ldots,\varphi_2(x;a_m)$ and $\phi_2(x;b_1),\ldots,\phi_2(x;b_m)$ are both $n$-inconsistent.
We prove this for $\varphi_2$, the same argument works for $\phi_2$.
Fix a subset $I$ of $\{1,\ldots,m\}$ such that $|I| = n$.
Let $U := \bigcap_{i \in I} \varphi_2(\Z;a_i)$ and $F := \bigcap_{i \in I} \varphi_1(\Z;a_i)$.
So $F \cap U$ is empty as $\varphi(x;a_1),\ldots,\varphi(x;a_m)$ is $n$-inconsistent.
Observe that $U$ is open and  $F = (k\Z + l) \setminus A$ for finite $A$.
So $F$ is dense in $\Z$ as $\valp(k) = 0$.
So $F \cap U$ is the intersection of a dense set and an open set, so $U$ is empty.
Thus $\varphi_2(x;a_1),\ldots\varphi_2(x;a_m)$ is $n$-inconsistent.
\end{proof}

\section{The $p$-adic completion}
\label{section:p-adic-completion}
\noindent
Among other things we show that $\pfield)$ is interpretable in the Shelah expansion of a highly saturated elementary extension of $\Sa P$, so $\Sa P$ is non-modular.
\newline

\noindent
We construct a $p$-adic completion $\Sq Z$ of a dp-minimal expansion $\Sa Z$ of $\zpp$.
We show that $\Sq Z$ is dp-minimal, but in contrast with the situation over $\zca$ we do not obtain an explicit description of unary definable sets.
So we first show that definable sets and functions in dp-minimal expansions of $\zp$ behave similarly to definable sets and functions in o-minimal structures.

\subsection{Dp-minimal expansions of $\zp$}
Let $\Sa Y$ expand $\zp$.

\begin{fact}
\label{fact:sw}
Suppose $\Sa Y$ is dp-minimal.
Then the following are satisfied for any $\Sa Y$-definable subset $X$ of $\Z_p^n$ and $\Sa Y$-definable function $f : X \to \Z_p^m$.
\begin{enumerate}
\item $X$ is a boolean combination of $\Sa Y$-definable closed subsets of $\Z_p^n$.
\item If $n = 1$ then $X$ is the union of a definable open set and a finite set.
\item The dp-rank of $X$, the $\acl$-dimension of $X$, and the maximal $0 \leq d \leq n$ for which there is a coordinate projection $\pi : \Z_p^m \to \Z_p^d$ such that $\pi(X)$ has interior are all equal.
(We denote the resulting dimension by $\dim X$.)
\item There is a $\Sa Y$-definable $Y \subseteq X$ such that $\dim X \setminus Y < \dim X$ and $f$ is continuous on $Y$.
\item The frontier inequality holds, i.e. $\dim \cl(X) \setminus X < \dim X$.
\end{enumerate}
Furthermore the same properties hold in any elementary extension of $\Sa Y$.
\end{fact}

\noindent
Fact~\ref{fact:sw} is a special case of the results of \cite{SW-tame}.
Every single item of Fact~\ref{fact:sw} fails in $\zpp$ because of the presence of dense and co-dense definable sets.
\newline


\noindent
There are dp-minimal expansions of valued groups in which algebraic closure does not satisfy the exchange property \cite{AV-couple,kuhlmann-contractions}, but this cannot happen over $\zp$

\begin{prop}
\label{prop:exchange}
Suppose $\Sa Y$ is dp-minimal.
Then $\Sa Y$ is a geometric structure, i.e. $\Sa Y$ eliminates $\exists^\infty$ and algebraic closure satisfies the exchange property.
\end{prop}

\begin{proof}
Elimination of $\exists^\infty$ follows from Fact~\ref{fact:exists-0}.
We show that algebraic closure satisfies exchange.
By \cite[Proposition 5.2]{SW-tame} exactly one of the following is satisfied,
\begin{enumerate}
\item algebraic closure satisfies exchange, or
\item there is definable open $U \subseteq \Z_p$, definable $F \subseteq U \times \Z_p$ such that each $F_a$ is finite, for every $a \in U$ there is an open $a \in V \subseteq U$ such that $F_b = F_a$ for all $b \in V$, and the family $(F_a : a \in U)$ contains infinitely many distinct sets.
\end{enumerate}
Suppose $(2)$ holds.
Let $E$ be the set of $(a,b) \in U^2$ such that $F_a = F_b$.
Then $E$ is a definable equivalence relation, every $E$-class is open, and there are infinitely many $E$-classes.
Suppose $A \subseteq U$ contains exactly one element from each $E$-class.
As $\Z_p$ is separable $|A| = \aleph_0$.
Let $D := \bigcup_{a \in U} F_a = \bigcup_{a \in A} F_a$.
So $D \subseteq \Z_p$ is definable and $|D| = \aleph_0$.
This contradicts Fact~\ref{fact:sw}$(2)$.
\end{proof}

\noindent
Finally, Fact~\ref{fact:p-min} is proven in \cite{SW-dp}.

\begin{fact}
\label{fact:p-min}
A dp-minimal expansion of $\pring)$ is $\pring)$-minimal.
\end{fact}

\noindent
It is an open question whether the theory of a dp-minimal expansion of $\pring)$ is $\Th\pring)$-minimal 
(equivalently: $P$-minimal).

\subsection{The $p$-adic completion}
Suppose $\Sa Z$ is an expansion of $\zpp$.
Let $\Sa S \prec \Sa N$ be highly saturated.
We define a standard part map $\st : N \to \Z_p$ by declaring $\st(a)$ to be the unique element of $\Z_p$ such that for all non-zero integers $k,k'$ we have $\valp(a - k) \geq k'$ if and only if $\valp(\st(a) - k) \geq k'$.
Note that $\st$ is a homomorphism and let $\minf$ be the kernal of $\st$.
We identify $N/\minf$ with $\Z_p$ and identify  $\st$ with the quotient map.
Note that $\minf$ is the set of $a \in N$ such that $\st(a) \geq k$ for all integers $k$, so $\minf$ is externally definable and we consider $\Z_p$ as an imaginary sort of $\Sh N$.

\begin{prop}
\label{prop:p-adic-complete}
Suppose $\Sa Z$ is $\nip$.
Then the following are interdefinable.
\begin{enumerate}
\item The structure $\Sq Z$ on $\Z_p$ with an $n$-ary relation symbol defining the closure in $\Z^n_p$ of every $\Sh Z$-definable subset of $\Z^n$.
\item The structure on $\Z_p$ with an $n$-ary relation symbol defining the image of each $\Sa N$-definable subset of $N^n$ under the standard part map $N^n \to \Z^n_p$.
\item The open core of the structure induced on $\Z_p$ by $\Sh N$.
\end{enumerate}
The structure induced on $\Z$ by $\Sq Z$ is a reduct of $\Sh Z$.
If $\Sa Z$ is dp-minimal then $\Sq Z$ is interdefinable with the structure induced on $\Z_p$ by $\Sh N$.
\end{prop}

\noindent
So in particular $\Sq Z$ is dp-minimal when $\Sa Z$ is dp-minimal.
All claims of Proposition~\ref{prop:p-adic-complete} except the last follow by easy alternations to the proof of Fact~\ref{fact:complete-0}.
\newline

\noindent
We prove the last claim of Proposition~\ref{prop:p-adic-complete}.

\begin{proof}
Suppose $\Sa Z$ is dp-minimal.
We show that $\Sq Z$ is interdefinable with the structure induced on $\Z_p$ by $\Sh N$.
It suffices to show that the induced structure on $\Z_p$ is interdefinable with its open core.
The structure induced on $\Z_p$ by $\Sh N$ is dp-minimal as $\Sh N$ is dp-minimal.
So by Fact~\ref{fact:sw} any $\Sh N$-definable set is a boolean combination of closed $\Sh N$-definable sets.
\end{proof}

\noindent
One can show that $\zpp^\square$ is interdefinable with $\zp$.
We omit this for the sake of brevity.
\newline

\noindent
We now give the $p$-adic analogue of Theorem~\ref{thm:complete-induced}.
The proof is essentially the same as that of Theorem~\ref{thm:complete-induced} so we leave the details to the reader.
(One applies Fact~\ref{fact:delon} at the same point that Fact~\ref{fact:ms} is applied in the proof of Theorem~\ref{thm:complete-induced}.)

\begin{prop}
\label{prop:p-adic-dense-pair}
Suppose that $A$ is a subset of $\Z_p^n$, $\pring,A)$ is $\nip$, and $\Th\pring)$ is an open core of $\Th\pring,A)$.
Let $\Sa A$ be the structure induced on $A$ by $\pring)$ and $X$ be the closure of $A$ in $\Z_p^n$.
Then 
\begin{enumerate}
\item The structure $\Sq A$ with domain $X$ and an $n$-ary relation for the closure in $X^n$ of each $\Sh A$-definable subset of $A^n$,
\item and the structure $\Sa X$ induced on $X$ by $\pring)$,
\end{enumerate}
are interdefinable.
(Note that $X$ is semialgebraic.)
\end{prop}


\noindent
Fact~\ref{fact:mari-main} and Proposition~\ref{prop:p-adic-dense-pair} together easily yield Proposition~\ref{prop:mari-completion}.

\begin{prop}
\label{prop:mari-completion}
The completion $\Sq P$ of $\Sa P$ is interdefinable with the structure induced on $\Z_p$ by $\pring)$ and $a \mapsto \pexp(pa)$.
So a subset of $\Z_p^n$ is $\Sq P$-definable if and only if it is of the form $\{ (a_1,\ldots,a_n) \in \Z_p^n : (\pexp(pa_1),\ldots,\pexp(pa_n)) \in X\}$ for a semialgebraic subset $X$ of $(\Z^{\times}_p)^n$.
\end{prop}

\noindent
Proposition~\ref{prop:mari-completion} shows that $\Sq P$ defines an isomorphic copy of $\pfield)$.
So if $\Sa P \prec \Sa N$ is highly saturated then $\Sh N$ interprets $\pfield)$, hence $\Sa P$ is non-modular.
We expect that $\Sa N$ does not interpret an infinite field, but we do not have a proof.

\subsection{A $p$-adic completion conjecture}
\begin{conj}
\label{conj:p-adic}
Suppose $\Sa Z$ is a dp-minimal expansion of $\zpp$.
Then the structure induced on $\Z$ by $\Sq Z$ is interdefinable with $\Sh Z$ and every $\Sh Z$-definable subset of $\Z^n$ is of the form $X \cap Y$ where $X$ is a $\Sq Z$-definable subset of $\Z^n_p$ and $Y$ is a $(\Z,+)$-definable subset of $\Z^n$.
\end{conj}

\noindent
The analogue of Conjecture~\ref{conj:p-adic} for dp-minimal expansions of divisible archimedean ordered groups is proven in \cite{SW-dp}.
We can prove a converse to Conjecture~\ref{conj:p-adic}.

\begin{prop}
\label{prop:p-adic-gen}
Let $\Sa Y$ be an expansion of $\zp$ and $\Sa Z$ be the structure induced on $\Z$ by $\Sa Y$.
Suppose $\Sa Y$ is dp-minimal and every $\Sa Z$-definable subset of $\Z^n$ is of the form $X \cap Y$ where $X$ is a $\Sa Y$-definable subset of $\Z^n_p$ and $Y$ is a $(\Z,+)$-definable subset of $\Z^n$.
Then $\Sa Z$ is dp-minimal.
\end{prop}

\begin{proof}
$\nip$ formulas are closed under conjunctions so $\Sa Z$ is $\nip$.
So it suffices to show that $\Sa Z$ is inp-minimal.
Inspection of the proof of Proposition~\ref{prop:p-adic-char} reveals that our proof on inp-minimality for $\Sa P$ only uses the following facts about $\pring)$:
\begin{enumerate}
    \item $\pring)$ is inp-minimal, and
    \item every definable unary set in every elementary extension of $\Z_p$ is the union of a finite set and a definable open set.
\end{enumerate}
It follows from Fact~\ref{fact:sw} that any dp-minimal expansion of $\zp$ satisfies $(2)$.
So the proof of Proposition~\ref{prop:p-adic-char} shows that $\Sa Z$ is inp-minimal.
\end{proof}

\section{$p$-adic elliptic curves?}
\label{section:elliptic}
\noindent
We give a conjectural construction of uncountably many dp-minimal expansions of $\zpp$.
Fix $\beta \in p\Z_p$.
Then $\beta^\Z$ is a closed subgroup of $\Q^{\times}_p$.
It is a well-known theorem of Tate~\cite{tate-elliptic} that there is an elliptic curve $\E_\beta$ defined over $\Q_p$ and a surjective $p$-adic analytic group homomorphism $\xi_\beta : \Q^{\times}_p \to \E_\beta(\Q_p)$ with kernel $\beta^\Z$.
Note that $\xi_\beta$ is injective on $1 + p\Z_p$ as $(1 + p\Z_p) \cap \beta^\Z = \{1\}$.
We let $\chi_\beta$ be the injective $p$-adic analytic homomorphism $(\Z_p,+) \to \E_\beta(\Q_p)$ given by  $\chi_\beta(a) := \xi_\beta(\pexp(pa))$, $\Sa Y_\beta$ be the structure induced on $\Z_p$ by $\pfield)$ and $\chi_\beta$, and $\Sa E_\beta$ be the structure induced on $\Z$ by $\pfield)$ and $\chi_\beta$.
So $\Sa E_\beta$ is the structure induced on $\Z$ by $\Sa Y_\beta$.

\begin{prop}
\label{prop:p-adic-elliptic}
$\Sa Y_\beta$ expands $\zp$ and $\Sa E_\beta$ expands $\zpp$.
\end{prop}

\noindent
Proposition~\ref{prop:p-adic-elliptic} requires some $p$-adic metric geometry.
We let
$$ \valp(a) = \min \{\valp(a_1),\ldots,\valp(a_m) \} \quad \text{for all  } a = (a_1,\ldots,a_m) \in \Q^m_p. $$
If $X,Y$ are subsets of $\Q^m_p$ then $f : X \to Y$ is an isometry if $f$ is a bijection and
$$ \valp( f(a) - f(a') ) = \valp(a - a') \quad \text{for all  } a,a' \in X. $$
Suppose $X,Y$ are $p$-adic analytic submanifolds of $\Q^m_p$.
We let $T_a X$ be the tangent space of $X$ at $a \in X$.
Given a $p$-adic analytic map $f : X \to Y$ we let $\jac f(a) : T_a X \to T_{f(a)} Y$ be the differential of $f$ at $a \in X$.

\begin{fact}
\label{fact:p-adic-isometry}
Suppose $f : X \to Y$ is a $p$-adic analytic map between $p$-adic analytic submanifolds $X,Y$ of $\Q^m_p$.
Fix $a \in X$ and set $b := f(a)$.
Suppose that $\jac f(a)$ is an isometry $T_a X \to T_b Y$.
Then there is an open neighbourhood $U$ of $p$ such that $f(U)$ is open and $f$ gives an isometry $U \to f(U)$.
\end{fact}

\noindent
See \cite[Proposition 7.1]{Glckner2006} for a proof of Fact~\ref{fact:p-adic-isometry} when $X,Y$ are open subsets of $\Q^m_p$.
This generalizes to $p$-adic analytic submanifolds as any $d$-dimensional $p$-adic analytic submanifold of $\Q^m_p$ is locally isometric to $\Q^d_p$, see for example \cite[5.2]{Halupczok} (Halupczok only discusses smooth $p$-adic algebraic sets but everything goes through for $p$-adic analytic submanifolds).
\newline

\noindent
We now prove Proposition~\ref{prop:p-adic-elliptic}.

\begin{proof}
To simplify notion we drop the subscript ``$\beta$".
It is enough to prove the first claim.
It is easy to see that $\Sa Y$ defines $+$.
We need to show that the set of $(a,a') \in \Z^2_p$ such that $\valp(a) \leq \valp(a')$ is definable in $\Sa E$.
Note that if $A$ is a finite subset of $\E(\Q_p)$ and $f : \E(\Q_p) \setminus A \to \Q^m_p$ is a semialgebraic injection then $\Sa E$ is interdefinable with the structure induced on $\Z_p$ by $\pfield)$ and $f \circ \chi$.
So we can replace $\E(\Q_p)$ and $\chi$ with $f(\E(\Q_p) \setminus A)$ and $f \circ \chi$.
\newline

\noindent
We consider $\E(\Q_p)$ as a subset of $\mathbb{P}^2(\Q_p)$ via the Weierstrass embedding.
Let $\iota : \Q^2_p \to \mathbb{P}^2(\Q_p)$ be the inclusion $\iota(a,a') = [a : a' : 1]$, $U$ be the image of $\iota$, and $E := \iota^{-1}(\E(\Q_p))$.
Recall that $\E(\Q_p) \setminus U$ is a singleton and $E$ is a $p$-adic analytic submanifold of $\Q^2_p$.
Let $\zeta : \Z_p \to E$ be $\zeta := \iota^{-1} \circ \chi$.
So $\Sa E$ is interdefinable with the structure induced on $\Z_p$ by $\pfield)$ and $\zeta$.
\newline

\noindent
Let $e := \zeta(0)$ and identify $T_0\Z_p$ with $\Q_p$.
Note that $\jac\zeta(0)$ is a bijection $\Q_p \to T_e E$.
After making an affine change of coordinates if necessary we suppose
$\jac \zeta(0)$ is an isometry $\Q_p \to T_e E$.
Applying Fact~\ref{fact:p-adic-isometry} we obtain $n$ such that the restriction of $\zeta$ to $p^n\Z_p$ is an isometry onto its image. 
So for all $a \in \Z_p$ we have
$$ \valp(\zeta(p^n a) - e) = \valp(p^n a - 0) =  \valp(a) + n.  $$
So for all $a,a' \in \Z_p$ we have
$$ \valp(a) \leq \valp(a') \quad \text{if and only if} \quad \valp(\zeta(p^n a) - e) \leq \valp(\zeta(p^n a') - e). $$
Let $X$ be the set of $(a,a') \in \Z^2$ such that $\valp(\zeta(a) - e) \leq \valp(\zeta(a') - e)$, so $X$ is definable in $\Sa E$.
So for all $(a,a') \in \Z^2_p$ we have $\valp(a) \leq \valp(a')$ if and only if $(p^n a, p^n a') \in X$.
So $\{ (a,a') \in \Z^2_p : \valp(a) \leq \valp(a') \}$ is definable in $\Sa Y$.
\end{proof}

\noindent
We denote the group operation on $\E_\beta(\Q_p)$ by $\oplus$.

\begin{conj}
\label{conj:tate}
Suppose $A$ is a finite rank subgroup of $\E_\beta(\Q_p)$.
Then $\pfield,A)$ is $\nip$, $\Th\pfield)$ is an open core of $\Th\pfield,A)$, and every $\pfield,A)$-definable subset of $A^k$ is of the form $X \cap Y$ where $X$ is an $(A,\oplus)$-definable subset of $A^k$ and $Y$ is a semialgebraic subset of $\E_\beta(\Q_p)^k$.
\end{conj}

\noindent
Suppose Conjecture~\ref{conj:tate} holds.
Under this assumption, Proposition~\ref{prop:p-adic-gen} shows that $\Sa E_\beta$ is dp-minimal, an application of Proposition~\ref{prop:p-adic-dense-pair} shows that $\Sa E_\beta^{\square}$ is interdefinable with $\Sa Y_\beta$, and an adaptation the proof of Theorem~\ref{thm:combine-1} shows that if $\Sa E_\alpha$ and $\Sa E_\beta$ are interdefinable then there is a semialgebraic group isomorphism $\E_\alpha(\Q_p) \to \E_\beta(\Q_p)$.
So we obtain an uncountable collection of dp-minimal expansions of $\zpp$ no two of which are interdefinable.
\newline

\noindent
Conjecture~\ref{conj:tate} should hold for any one-dimensional $p$-adic semialgebraic group satisfying a Mordell-Lang condition.
One dimensional $p$-adic semialgebraic groups are classified in ~\cite{lopez-one-dim}.
\newline

\noindent
Suppose that $\bH$ is a one-dimensional $\pfield)$-definable group.
By \cite{HrushovskiPillay} there is an open subgroup $V$ of $\bH$, a one-dimensional abelian algebraic group $W$ defined over $\Q_p$, an open subgroup $U$ of $W(\Q_p)$, and a $\pfield)$-definable group isomorphism $V \to U$.
So we suppose that $\bH$ is $W(\Q_p)$, so in particular $\bH$ is a $p$-adic analytic group.
Let $e$ be the identity of $\bH$ and identify $T_e\bH$ with $\Q_p$.
For sufficiently large $n$ there is an open subgroup $U$ of $\bH$ and a $p$-adic analytic group isomorphism $\Xi : (p^n\Z_p,+) \to U$, this $\Xi$ is the Lie-theoretic exponential, see \cite[Corollary 19.9]{p-adic-lie-groups}.
Let $\Sa H$ be the structure induced on $\Z$ by $\pfield)$ and $k \mapsto \Xi(p^n k)$.
It follows in the same way as above that $\Sa H$ expands $\zpp$.
We expect that if $\bH$ is semiabelian then $\Sa H$ is dp-minimal and $\Sq H$ is interdefinable with the structure induced on $\Z_p$ by $\pfield)$ and $a \mapsto \Xi(p^n a)$.

\section{A general question}
\label{section:gen-ques}
\noindent
We briefly discuss the following question raised to us by Simon: Is there an abstract approach to $\Sq Z$?
There are many ways in which one might try to make this more precise.
For example: Given a sufficiently well behaved $\nip$ structure $\Sa M$ (perhaps dp-minimal, perhaps distal, perhaps expanding a group) can one construct a canonical structure $\Sq M$ containing $\Sa M$ such that $\Sh M$ is the structure induced on $M$ by $\Sq M$, $\Sq M$ is somehow ``close to o-minimal", and $\Sq M$ is not too ``big" relative to $\Sa M$?
In the completion of an $\nip$ expansion $\Sa H$ of an archimedean ordered abelian group, $\mfin$ is $\bigvee$-definable, $\minf$ is $\bigwedge$-definable, and the resulting logic topology on $\R$ agrees with the usual topology.
The same thing happens for the other completions discussed above.
So perhaps there is a highly saturated $\Sa M \prec \Sa N$, a set $X$ which both externally definable and $\bigvee$-definable in $\Sa N$, an equivalence relation $E$ on $X$ which is both externally definable and $\bigwedge$-definable in $\Sa N$, such that $\Sq M$ is the structure induced on $X/E$ by $\Sh N$.
\newline

\noindent
The completions defined above are not always the ``right" notion.
Let $P$ be the set of primes and fix $q \in P$.
Consider $(\Z,+,(\valp )_{ p \in P})$ as an expansion of $(\Z,+,\valq)$, one can show that $(\Z,+,(\valp)_{p \in P})^\square$ is interdefinable with $\zq$.
However the ``right" completion of $(\Z,+,(\valp )_{ p \in P} )$ is $(\widehat{\Z},+,(\triangleleft_p)_{ p \in P})$ where $(\widehat{\Z},+)$ is the profinite completion $\prod_{p \in P} (\Z_p,+)$ of $(\Z,+)$ and we have $a \triangleleft_p b$ if and only if $\valp(\pi_p(a)) < \valp(\pi_p(b))$, where $\pi_p$ is the projection $\widehat{\Z} \to \Z_p$.
Likewise, if $I$ is a $\Z$-linearly independent subset of $\R \setminus \Q$ then the completion of $(\Z,+,(C_\alpha)_{\alpha \in I})$ should be the torus $( (\rz)^I,+,S_\alpha)$ where we have $S_\alpha(a,b,c)$ if and only if $C(\pi_\alpha(a),\pi_\alpha(b),\pi_\alpha(c))$ where $\pi_\alpha$ is the projection $(\rz)^I \to \rz$ onto the $\alpha$th coordinate.
\newline

\noindent
When $\Sa M$ expands a group it is tempting to try to define $\Sq M$ via general ideas from $\nip$ group theory.
Fix irrational $\alpha \in \rz$, let $\Sa Z$ be a dp-minimal expansion of $\zca$, and $\Sa Z \prec \Sa N$ be highly saturated.
One can show that $\Sa N^{0}/\Sa N^{00}$ is isomorphic as a topological group to $\rz$ and it seems likely that $\Sq Z$ is interdefinable with the structure on $\rz$ with an $n$-ary relation defining the image of $X \cap (\Sa N^0)^n$ under the quotient map $(\Sa N^0)^n \to (\rz)^n$, for each $\Sa N$-definable $X \subseteq N^n$.
But this breaks in the $p$-adic setting, if $\zpp \prec \Sa N$ is highly saturated then $\Sa N^{0} = \Sa N^{00}$.
\newline

\noindent
Is there some general class of ``complete" structures for which $\Sa M$ and $\Sq M$ are interdefinable?
Suppose $H$ is a dense subgroup of $(\R,+)$ and $\Sa H$ is an $\nip$ expansion of $(H,+,<)$.
Then $\Sa H$ and $\Sq H$ are interdefinable if and only if $H = \R$ and $\Sa H$ is interdefinable with the open core of $\Sh H$.
The Marker-Steinhorn theorem shows that these conditions are satisfied when $\Sa H$ is an o-minimal expansion of $(\R,+,<)$.
If $\Sa H$ is o-minimal then $\Sq H$ is (up to interdefinibility) the unique elementary extension of $\Sa H$ which expands $(\R,+,<)$ \cite[13.2.1]{big-nip}.
(Laskowski and Steinhorn~\cite{LasStein} showed that there is such an extension.)
If $\Sa H$ is weakly o-minimal then $\Sq H$ is an elementary expansion of $\Sa H$ if and only if $\Sa H$ is o-minimal.
Should $\Sa M$ be ``complete" if $\Sq M$ is (up to interdefinability) an elementary extension of $\Sa M$?
If $\Sa Z$ is a dp-minimal expansion of $\zp$ then must $\Sq Z$ and $\Sa Z$ be interdefinable, i.e. is there a $p$-adic Marker-Steinhorn generalizing Fact~\ref{fact:delon}?
\newline

\bibliographystyle{abbrv}
\bibliography{NIP}
\end{document}